\documentclass[10pt]{amsart}
\usepackage{amssymb}

\usepackage{enumerate}

\usepackage{graphicx}

\makeatletter
\@namedef{subjclassname@2010}{%
  \textup{2010} Mathematics Subject Classification}
\makeatother

\usepackage[utf8x]{inputenc}
\usepackage[all]{xy}
\usepackage{ucs}
\usepackage{amsmath}
\usepackage{amsfonts}
\usepackage{amssymb}
\usepackage{amsthm}
\usepackage{graphicx}

\newtheorem{thm}{Theorem}[section]
\newtheorem{cor}[thm]{Corollary}
\newtheorem{lem}[thm]{Lemma}
\newtheorem{prop}[thm]{Proposition}

\theoremstyle{definition}
\newtheorem{defin}[thm]{Definition}
\newtheorem{rem}[thm]{Remark}
\newtheorem{exa}[thm]{Example}

\newtheorem*{xrem}{Remark}

\newcommand{\C}{\mathbb{C}}

\newcommand{\biindice}[3]%
{%

\begin{array}[t]{c}
{\displaystyle #1}\\
{\scriptstyle #2}\\
{\scriptstyle #3}
\end{array}

}

\newcommand\ind{{\rm 1\kern-.30em I}}
\newcommand\dis{\displaystyle}
\newcommand\eps{\varepsilon}

\begin{document}

\title[]{Lacunary M\"untz spaces: \\ isomorphisms and Carleson embeddings}

\author[Lo\"ic Gaillard]{Lo\"ic Gaillard}
\address{Laboratoire de Math\'ematiques de Lens (LML), EA 2462, F\'ed\'eration CNRS Nord-Pas-de-Calais FR~2956, Universit\'e d'Artois, rue Jean Souvraz S.P. 18\\ 62307 Lens Cedex, France}
\email{loic.gaillard@univ-artois.fr}

\author[Pascal Lef\`evre]{Pascal Lef\`evre}
\address{Laboratoire de Math\'ematiques de Lens (LML), EA 2462, F\'ed\'eration CNRS Nord-Pas-de-Calais FR~2956, Universit\'e d'Artois, rue Jean Souvraz S.P. 18\\ 62307 Lens Cedex, France}
\email{pascal.lefevre@univ-artois.fr}

\date{}

\begin{abstract}
In this paper we prove that $M_\Lambda^p$ is almost isometric to $\ell^p$ in the canonical way when $\Lambda$ is lacunary with a large ratio. 
 On the other hand, our approach can be used to study also the Carleson measures for M\"untz spaces $M_\Lambda^p$ when $\Lambda$ is lacunary. We give some necessary and some sufficient conditions to ensure that a Carleson embedding is bounded or compact. In the hilbertian case, the membership to Schatten classes is also studied. When $\Lambda$ behaves like a geometric sequence the results are sharp, and we get some characterizations.
\end{abstract}

\subjclass[2010]{30B10, 47B10, 47B38.}

\keywords{M\"untz spaces, Carleson embeddings, lacunary sequences, Schatten classes.}

\maketitle

\section{Introduction}

Let $m$ be the Lebesgue measure on $[0,1]$. For $p\in[1,+\infty)$, $L^p(m)=L^p([0,1],m)$ (sometimes denoted simply $L^p$ when there is no ambiguity) denotes the space of complex-valued measurable functions on $[0,1]$, equipped with the norm $\|f\|_p=(\int_0^1|f(t)|^pdt)^{\frac{1}{p}}.$ In the same way, $\mathcal{C}=C([0,1])$ is the space of continuous functions on $[0,1]$ equipped with the usual sup-norm.
 We shall also consider some positive and finite measures $\mu$ on $[0,1)$ (see the remark at the beginning of section 2), and the associated $L^p(\mu)$ space.
For a sequence $w=(w_n)_n$ of positive weights, we denote $\ell^p(w)$ the Banach space of complex sequences $(b_n)_n$ equipped with the norm $\|b\|_{\ell^p(w)}=(\sum_n|b_n|^pw_n)^{\frac{1}{p}}$ and the vector space $c_{00}$ consisting on complex sequences with a finite number of non-zero terms.
 All along the paper, when $p\in(1,+\infty)$, we denote as usual $p'=\tfrac{p}{p-1}$ its conjugate exponent.

The famous M\"untz theorem (\cite[p.172]{BE},\cite[p.77]{GL}) states that if $\Lambda=(\lambda_n)_{n\in\mathbb N}$ is an increasing sequence of non-negative real numbers, then the linear span of the monomials $t^{\lambda_n}$ is dense in $L^p$ (resp. in $\mathcal{C}$) if and only if $\sum_{n\geq 1}\tfrac{1}{\lambda_n}=+\infty$ (resp. and $\lambda_0=0$). 
We shall assume that the M\"untz condition $\sum_{n\geq 1}\tfrac{1}{\lambda_n}<+\infty$ is fulfilled
and we define the M\"untz space $M_\Lambda^p$ as the closed linear space spanned by the monomials $t^{\lambda_n}$, where $n\in\mathbb{N}.$
We shall moreover assume that $\Lambda$ satisfies the gap condition: $\inf\limits_n \big(\lambda_{n+1}-\lambda_n\big)>0$. Under this later assumption the Clarkson-Erd\"os theorem holds \cite[Th.6.2.3]{GL}: the functions in $M_\Lambda^p$ are the functions $f$ in $L^p$ such that $f(x)=\sum a_nx^{\lambda_n}$ (pointwise on $[0,1)$).
This gives us a class of Banach spaces $M_\Lambda^p\subsetneq L^p$ of analytic functions on $(0,1)$. 

In full generality, the M\"untz spaces are difficult to study, but for some particular sequences $\Lambda$, we can find some interesting properties of the spaces $M_\Lambda^p$. Let us mention that lately these spaces received an increasing attention from the point of view of their geometry and operators: the monograph of Gurariy-Lusky \cite{GL}, and various more or less recent papers (see for instance \cite{AHLM},\cite{AL},\cite{CFT},\cite{LL},\cite{NT}).

We shall focus on two different questions on the M\"untz spaces. The first one is linked to an old result: Gurariy and Macaev proved in \cite{GM} that, in $L^p$, the normalized sequence $((p\lambda_n+1)^{\frac{1}{p}}t^{\lambda_n})_n$ is equivalent to the canonical basis of $\ell^p$ if and only if $\Lambda$ is lacunary (see Th.\ref{thm gurariy} below). More recently, the monograph \cite{GL} introduces the notion of quasi-lacunary sequence (see definition \ref{def lac} below), and states that $M_\Lambda^p$ is still isomorphic to $\ell^p$ when $\Lambda$ is quasi-lacunary. On the other hand, some recent papers discuss about the Carleson measures for the M\"untz spaces. In  \cite{CFT}, the authors introduced the class of sublinear measures on $[0,1)$, and proved that when $\Lambda$ is quasi-lacunary, the sublinear measures are Carleson embeddings for $M_\Lambda^1$.  In \cite{NT}, the authors extended this result to the case $p=2$ but only when the  sequence $\Lambda$ is lacunary.
 
In this paper, we introduce another method to study the lacunary M\"untz spaces: for a weight $w$ and a measure $\mu$ on $[0,1)$, we define  $T_\mu:\ell^p(w)\rightarrow L^p(\mu)$ by $T_\mu(b)=\sum_n b_nt^{\lambda_n}$ for $b=(b_n)\in\ell^p(w)$. The operator $T_\mu$ depends on $w,\mu,p$ and $\Lambda$, and when it is bounded we shall denote by $\|T_\mu\|_p$ its norm. We shall see that an estimation of $\|T_\mu\|_p$ can be used to improve the theorem of Gurariy-Macaev, and to generalize former Carleson embedding results to lacunary M\"untz spaces $M_\Lambda^p$ for any $p\ge1$.

The paper is organized as follows: in part 2, we specify the missing notations and some usefull lemmas. The main result gives an upper bound for the approximation numbers of $T_\mu$ (see Prop.\ref{prop an(T) Dn}). 
In section 3, we focus on the classical case: we fix $w_n=(p\lambda_n+1)^{-1}$ and we define $J_\Lambda:\ell^p(w)\rightarrow M_\Lambda^p$ by $J_\Lambda(b)=\sum_n b_nt^{\lambda_n}$. It is the isomorphism underlying in the theorem of Gurariy-Macaev. For $p>1$, we prove that $J_\Lambda$ is bounded exactly when $\Lambda$ is quasi-lacunary. On the other hand, when $\Lambda$ is lacunary with a large ratio, we also get a sharp bound for $\|J_\Lambda^{-1}\|_p$  (see Th.\ref{thm r epsilon} below). 
Our approach leads to an asymptotically orthogonal version of Gurariy-Macaev theorem exactly for the super-lacunary sequences.
In section 4, we apply the results of section 2 for a positive and finite measure $\mu$ on $[0,1)$ with the weights $w_n=\lambda_n^{-1}$.
To treat the Carleson embedding problem, we shall give an estimation of the approximation numbers of the embedding operator $i_\mu^p:M_\Lambda^p\rightarrow L^p(\mu)$. In section 5, we focus on the compactness of $i_\mu^p$ using the same tools as in section 4. In the case $p=2$, this leads to some control of the Schatten norm of the Carleson embedding and some characterizations when $\Lambda$ behaves like a geometric sequence.\medskip

As usual the notation, $A\lesssim B$ means that there exists a constant $c>0$ such that $A\le cB$. This constant $c$ may depend along the paper on $\Lambda$ (or sometimes only on its ratio of lacunarity), on $p\ldots$. We shall specify this dependence to avoid any ambiguous statement. In the same way, we shall use the notations $A\approx B$ or $A\gtrsim B$.

\section{Preliminary results}

 Before giving preliminary results, let us give a few words of explanation about our choice of measures on $[0,1)$. This comes from the fact that the measures involved (if considered on $[0,1]$) must satisfy $\mu(\{1\})=0$. Indeed, we focus either on the Lebesgue measure $m$ (satisfying of course $m(\{1\})=0$) or on measures such that the Carleson embedding $f\in M_\Lambda^p\mapsto f\in L^p(\mu)$ is (defined and) bounded, so that testing a sequence of monomials $g_n(t)=t^{\lambda_n}$ we must have $$\dis\mu(\{1\})=\lim\|g_n\|^p_{L^p(\mu)}\lesssim\lim\|g_n\|^p_{L^p(m)}=0.$$  Therefore practically, we shall consider in the whole paper measures on $[0,1)$. Moreover, thanks to the result of Clarkson-Erd\"os, the value  at any point of $[0,1)$ of any function of a M\"untz space can be defined without ambiguity.

\medskip

We shall need several notions of growth for increasing sequences.

\begin{defin}\label{def lac}

\begin{itemize}

\item A sequence $u=(u_n)_n$ of  positive numbers is said to be {\em lacunary} if there exists $r>1$ such that $\displaystyle {u_{n+1}}\ge r{u_n}$, for every $n\in\mathbb{N}$. We shall say that such a sequence is $r$-lacunary and that $r$ is a ratio of lacunarity of this sequence.

\item The sequence $u$ is called {\em quasi-lacunary} if there is an extraction $(n_k)_k$ such that $\sup\limits_{k\in\mathbb{N}}(n_{k+1}-n_k)<+\infty,$ and $(u_{n_k})_k$ is lacunary.

\item The sequence $u$ is called {\em quasi-geometric} if there are two constants $r$ and $R$ such that we have $\displaystyle 1<r\leq \frac{u_{n+1}}{u_n}\leq R<+\infty$, for every $n\in\mathbb{N}$. In particular, these sequences are lacunary.

\item The sequence $u$ is called {\em super-lacunary} if $\displaystyle  \frac{u_{n+1}}{u_n}\longrightarrow+\infty$.
\end{itemize}
\end{defin}

\begin{rem} \label{rem qlac fini}
It is proved in \cite[Prop.7.1.3 p.94]{GL} that a sequence is quasi-lacunary if and only if it is a finite union of lacunary sequences.
\end{rem}

The following result is due to Gurariy and Macaev.

\begin{thm}\label{thm gurariy}{\cite[Corollary 9.3.4, p.132]{GL}}

For $p\in[1,+\infty)$, the following are equivalent:
\begin{enumerate}[(i)]
\item The sequence $\Lambda$ is lacunary.
\item The sequence $\Big(\displaystyle\frac{t^{\lambda_n}}{\|t^{\lambda_n}\|_{p}}\Big)$ in $L^p$ is equivalent to the canonical basis of $\ell^p$.
\end{enumerate}
In particular, since $\|t^{\lambda_n}\|_{p}=(p\lambda_n+1)^{-\frac{1}{p}}$, we have for any $b\in c_{00}$
$$\Big\|\sum b_nt^{\lambda_n}\Big\|_{p} \approx \Big(\sum\frac{|b_n|^p}{p\lambda_n+1}\Big)^{\frac{1}{p}}$$ when $\Lambda$ is lacunary, and where the underlying constants depend on $p$ and $\Lambda$ only.
\end{thm}

We shall recover and generalize partially this result: for a given sequence of weights $(w_n)_n$ and a positive finite measure $\mu$ on $[0,1)$, we study the boundedness of the operator
$$T_\mu:\left\lbrace\begin{array}{ccc}
\ell^p(w)&\longrightarrow &L^p(\mu)\\
b &\longmapsto & \sum b_nt^{\lambda_n}
\end{array}\right.~~.$$

\begin{exa}
In the case of the Lebesgue measure $\mu=m$ and when the weights are $w_n=(p\lambda_n+1)^{-1}$ or in a simpler way (when we do not care on the value of the constants) $w_n=\lambda_n^{-1}$, Th.\ref{thm gurariy} states in particular that $T_m$ is bounded when $\Lambda$ is lacunary.
\end{exa}

\begin{rem}\label{rem bornebrut}

In the case $p>1$, a (rough) sufficient condition to ensure the boundedness of $T$ is
$$\displaystyle\int_{[0,1)}\Big(\sum_n w_n^{-\frac{p'}{p}} t^{p'\lambda_n}\Big)^\frac{p}{p'}d\mu<\infty. $$
Indeed, this is just the consequence of the majorization 
$$\biindice{\sup}{b\in B_{\ell^{p}}}{b\in c_{00}}\sup_{g\in B_{L^{p'}(\mu)}}\Big|\int_{[0,1)}\sum_n b_nw_n^{-\frac{1}{p}} t^{\lambda_n}g(t)\,d\mu\Big|\le\sup_{g\in B_{L^{p'}(\mu)}}\int_{[0,1)}|g(t)|\biindice{\sup}{b\in B_{\ell^{p}}}{b\in c_{00}}\Big|\sum_n b_nw_n^{-\frac{1}{p}} t^{\lambda_n}\Big|\,d\mu. $$

Point out that in the case of standard weights $w_n\approx \lambda_n^{-1}$ and for a quasi-geometric sequence $\Lambda$, this condition can be reformulated with the help of Lemma~\ref{lem 1 sur 1-t} below as 
$$\int_{[0,1)}\frac{1}{1-t}~d\mu\approx\int_{[0,1)}\frac{1}{1-t^{p'}}\,d\mu<\infty$$ 
but we shall come back to that kind of condition later (see Prop.\ref{prop orderbounded} below for instance).
\end{rem}

To get a sharper estimation, we introduce the sequence $(D_n(p))_n$ defined for $n\in\mathbb{N}$ and $p\geq 1$, with a priori value in $\mathbb{R}_+\cup \{+\infty\}$ by
\begin{align*}
D_n(p)=\Bigg(\displaystyle\int_{[0,1)}w_n^{-\frac{1}{p}} t^{\lambda_n}\Bigg(\sum\limits_{k\ge0}w_k^{-\frac{1}{p}}t^{\lambda_k}\Bigg)^{p-1}d\mu\Bigg)^{\frac{1}{p}}.
\end{align*}

\begin{prop}\label{prop w mu Dn}
Let $p\in[1,+\infty)$. Assume that $(D_n(p))_n$ is a bounded sequence of real numbers.
Then we have for every  $b\in\ell^p(w)$,
 $$\Big\|\sum\limits_{n\geq 0} b_nt^{\lambda_n}\Big\|_{L^p(\mu)}\leq \Big(\sum\limits_{n\geq 0}|b_n|^pw_nD_n(p)^p\Big)^{\frac{1}{p}}~~.$$
\end{prop}
\begin{proof}
If $p=1$ the result is obvious. Assume now that $p>1$. For any $t\in[0,1)$ and $n\in\mathbb{N}$, we have:
$$b_nt^{\lambda_n}=b_nw_n^{\frac{1}{pp'}}t^{\frac{\lambda_n}{p}}\times w_n^{-\frac{1}{pp'}}t^{\frac{\lambda_n}{p'}}~~,$$
we apply Hölder's inequality and get:
$$\Big|\sum b_n t^{\lambda_n}\Big|\leq \Big(\sum\limits_n |b_n|^pw_n^{\frac{1}{p'}}t^{\lambda_n} \Big)^{\frac{1}{p}}  \Big(\sum\limits_k w_k^{-\frac{1}{p}}t^{\lambda_k} \Big)^{\frac{1}{p'}}~~.$$
We obtain:
\begin{align*}
\int_{[0,1)}\Big|\sum b_nt^{\lambda_n}\Big|^pd\mu &\leq \int_{[0,1)}\sum |b_n|^pw_n . w_n^{-\frac{1}{p}}t^{\lambda_n}\Big( \sum\limits_k w_k^{-\frac{1}{p}}t^{\lambda_k} \Big)^{p-1}d\mu\\
&=\sum\limits_n |b_n|^pw_nD_n(p)^p~~.
\end{align*}
\end{proof}

If $(D_n(p))_n$ is a bounded sequence of real numbers, 
we define the bounded diagonal operator  $$\mathcal{D}:\ell^p(w)\to \ell^p(w)$$ acting on the canonical basis of $\ell^p(w)$ whose diagonal entries are the numbers $D_n(p)$. In other words, in that case, $T_\mu$ and $\mathcal{D}$ are bounded, and we have$$\forall b\in\ell^p(w),~\|T_\mu(b)\|_{L^p(\mu)}\leq \|\mathcal{D}(b)\|_{\ell^p(w)}~~.$$
This gives informations about the approximation numbers of $T_\mu$. Let us specify this notion. We shall be interested in how far from compact (the essential norm) or, on the contrary, how strongly compact (possibly Schatten in the Hilbert framework) are the Carleson embeddings. A way to measure this is to estimate the approximation numbers:

\begin{defin}
For a bounded operator $S:X\rightarrow Y$ between two separable Banach spaces $X,Y$,
the {\em approximation numbers} $(a_n(S))_n$ of $S$ are defined for $n\geq 1$ by 
$$a_{n}(S)=\inf\{\|S-R\|,rank(R)<n\}~~.$$

The {\em essential norm} of $S$ is defined by $$\|S\|_e=\inf\{\|S-K\|,K\text{ compact}\}~~.$$It is the distance from $S$ to the compact operators. 
\end{defin}

We shall use in the sequel the following notions of operator ideal. 

\begin{defin}\quad

\begin{itemize}
\item An operator $S:X\rightarrow Y$ is {\em nuclear} if there is a sequence of rank-one operators $(R_n)$ satisfying $S(x)=\sum\limits_n R_n(x)$ for every $x\in X$ with $\sum\limits_n\|R_n\|<+\infty.$
The {\em nuclear norm} of $S$ is defined as $$\|S\|_{\mathcal{N}}=\inf\Big\{\sum\limits_{n}\|R_n\|,rank(R_n)=1,\sum\limits_n R_n=S\Big\}~~.$$

\item An operator $S:X\rightarrow L^p(\mu)$ is {\em order bounded} if there exists a positive function $h\in L^p(\mu)$ such that for every $x\in B_X$ and for $\mu-$almost every $t\in \Omega$ we have $|S(x)(t)|\leq h(t).$

\item For $r>0$ and when $X, Y$ are Hilbert spaces, we say that a (compact) operator $S:X\rightarrow Y$ belongs to the {\em Schatten class} $\mathcal{S}^r$ if $$\sum\limits_{n}a_n(T)^r<+\infty.$$ In this case, we define its {\em Schatten norm} by $\|S\|_{\mathcal{S}^r}=\Big(\sum\limits_na_n(S)^r\Big)^{\frac{1}{r}}$. 

\end{itemize}

\end{defin}

Recall that nuclear and Schatten class operators are always compact.

Of course, the Schatten norm is really a norm when $r\geq 1$. The $\mathcal{S}^2$ class is also called the class of {\em Hilbert-Schmidt} operators. 
\medskip

For technical reasons, we introduce the following notation: for a bounded sequence $(u_n)_n$ in $\mathbb{R}_+$, we define $(u_N^\ast)_N$ the {\em decreasing rearrangement} of $(u_n)_n$ by $$u_N^\ast=\biindice{\inf}{A\subset\mathbb{N}}{|A|=N}\sup\{u_n,n\not\in A\}~~.$$
We have $\lim\limits_{N\rightarrow+\infty}u_N^\ast=\limsup\limits_{n\rightarrow+\infty}u_n.$

Now, we can state, 

\begin{prop}\label{prop an(T) Dn}
If $(D_n(p))_n$ is a bounded sequence of real numbers, then we have
\begin{enumerate}[(i)]
\item $a_{N+1}(T_\mu)\leq D_N(p)^\ast.$
\item $\|T_\mu\|_{p}\leq \sup\limits_{n\in\mathbb{N}}D_n(p)$.
\item $\|T_\mu\|_e\leq \limsup\limits_{n\rightarrow+\infty} D_n(p)$.
\item $\forall p\ge1$, $\dis\|T_\mu\|_{\mathcal N}\le\sum\limits_{n\ge0} w_n^{-\frac{1}{p}}\big\|t^{\lambda_n}\big\|_{L^p(\mu)}\,.$
\item If $p=2$, then for any $r>0$, $~~\|T_\mu\|_{\mathcal{S}^r}\leq \Big(\sum\limits_{n\ge0}D_n(2)^r\Big)^{\frac{1}{r}}~~.$
\end{enumerate}
\end{prop}
\begin{proof}
We first prove $(i)$. For $n\in\mathbb{N}$, we denote $\varphi_n^\ast:\ell^p(w)\rightarrow\C$ the functional on $\ell^p(w)$ defined by $\varphi_n^\ast(u)=u_n$ for a sequence $u=(u_n)_n\in\ell^p(w)$. We define also $g_n\in L^p(\mu)$ by $g_n(t)=t^{\lambda_n}$.
For any integer $N$ and $A\subset\mathbb{N}$ with $|A|=N$, we have:
\begin{align*}
a_{N+1}(T_\mu)\leq \Big\|T_\mu-\sum\limits_{n\in A}\varphi_n^\ast\otimes g_n\Big\|~~.
\end{align*}
We fix $b\in\ell^p(w)$ and apply Prop.\ref{prop w mu Dn}:
$$\Big\|T_\mu(b)-\sum\limits_{n\in A}\varphi_n^\ast (b) g_n\Big\|=\Big\|\sum\limits_{n\not\in A}b_nt^{\lambda_n}\Big\|_{L^p(\mu)}\leq \sup\limits_{n\not\in A}D_n(p)\|b\|_{\ell^p(w)}$$
and so $(i)$ holds. 

The points $(ii)$ and $(iii)$ are direct consequences of $(i)$.

The assertion $(iv)$ follows easily from the natural decomposition $T_\mu(b)=\sum\limits_n\varphi_n^\ast(b)t^{\lambda_n}$ and the fact that $\|\varphi^\ast_n\|=w_n^{-\frac{1}{p}}$.

For $(v)$: if $(D_n(2))_n\not\in\ell^r$ then the result is obvious. Else, we have in particular $D_n(2)\rightarrow 0$ when $n\rightarrow+\infty$. Since for all $\varepsilon>0$, the set $\{n,D_n(2)\geq \varepsilon\}$ is finite, there exists a bijection $\varphi:\mathbb{N}\rightarrow\mathbb{N}$ such that for any $n\in\mathbb{N}$, $D_n(2)^\ast=D_{\varphi(n)}(2)$. We have:
$$\sum\limits_{N}a_{N+1}(T_\mu)^r\leq \sum\limits_{N} (D_N(2)^{\ast})^r=\sum\limits_{n}D_{\varphi(n)}(2)^r=\sum\limits_{n}D_n(2)^r.$$
\end{proof}

\begin{lem}\label{lem 1 sur 1-t}
Let $\alpha\in\mathbb{R}_+^\ast$.
Assume that $\Lambda$ is a quasi-geometric sequence.
 Then there are two constants $C_1,C_2\in\mathbb{R}_+^\ast$ such that for any $t\in[0,1)$ we have:
$$C_1\Big(\frac{1}{1-t}\Big)^\alpha\leq \sum\limits_{n}\lambda_n^\alpha t^{\lambda_n}\leq C_2\Big(\frac{1}{1-t}\Big)^\alpha\cdot$$
\end{lem}
\begin{proof}
Since $\Lambda$ is quasi-geometric, it is $r$-lacunary for some $r>1$, so there exists a constant $C=(r-1)^{-1}$ such that for any $n\in\mathbb{N},\lambda_n\leq C(\lambda_{n+1}-\lambda_n)$. Moreover, there is a constant $R>1$ such that $\lambda_{n+1}\leq  R\lambda_n$ and hence we have:
$$\lambda_n^{\alpha}\approx(\lambda_{n+1}-\lambda_n)^\alpha\approx \lambda_{n+1}^\alpha$$
where the underlying constants do not depend on $n$. We obtain:
\begin{align*}
\sum\limits_n\lambda_n^{\alpha}t^{\lambda_n}&\approx \sum\limits_n(\lambda_{n+1}-\lambda_n)^{\alpha}t^{\lambda_n}
\approx \sum\limits_n\sum\limits_{\lambda_n\leq m<\lambda_{n+1}}(\lambda_{n+1}-\lambda_n)^{\alpha-1}t^{\lambda_n}\\
&\approx \sum\limits_n\sum\limits_{\lambda_n\leq m<\lambda_{n+1}} m^{\alpha-1}t^{\lambda_n}
\end{align*}
For $m$ such that $\lambda_n\leq m<\lambda_{n+1}$, we have $t^m\lesssim t^{\lambda_n}\lesssim t^{\frac{m}{R}}$ and so we obtain:
$$\sum\limits_n\lambda_n^\alpha t^{\lambda_n}\lesssim \sum\limits_{m\ge0} m^{\alpha-1}t^{\frac{m}{R}}\lesssim \Big(\frac{1}{1-t^{\frac{1}{R}}}\Big)^{\alpha}\lesssim \Big(\frac{1}{1-t}\Big)^{\alpha}\cdot$$
On the other hand we have
$$\sum\limits_n\lambda_n^\alpha t^{\lambda_n}\gtrsim \sum\limits_{m\in\mathbb{N}} m^{\alpha-1}t^m\gtrsim \Big(\frac{1}{1-t}\Big)^{\alpha}\cdot$$
\end{proof}

\begin{rem}\label{rem 1 sur 1-t}
If $\Lambda$ is only lacunary, the majorization part of the result above still holds.
Indeed, the proof above can be easily adapted, but anyway, we can also notice that there  exists a quasi-geometric sequence $\Lambda'=(\lambda_n')_n$ which contains $\Lambda$, and we have $$\sum\limits_{n\in\mathbb{N}}\lambda_n^{\alpha}t^{\lambda_n}\leq \sum\limits_{n\in\mathbb{N}}\lambda_n'^{\alpha}t^{\lambda_n'}\leq C_2\frac{1}{(1-t)^{\alpha}}\cdot$$
\end{rem}

We can give a new proof of the majorization part of the theorem of Gurariy-Macaev (Th.\ref{thm gurariy}). It follows from the next proposition:

\begin{prop}\label{prop gurariy rapide}
Let $p\in[1,+\infty)$. Assume that the weights are given by $w_n=\lambda_n^{-1}$ or $(p\lambda_n+1)^{-1}$. If $\Lambda$ is lacunary and $\mu$ is the Lebesgue measure, then 
 $(D_n(p))_n$ is a bounded sequence.
\end{prop}
\begin{proof}
 From Lemma~\ref{lem 1 sur 1-t} and Remark~\ref{rem 1 sur 1-t} we get:
\begin{align*}
D_n(p)^p& =\lambda_n^{\frac{1}{p}}\int t^{\lambda_n}\Big(\sum\limits_{k\in\mathbb{N}}\lambda_k^{\frac{1}{p}}t^{\lambda_k}\Big)^{p-1}dt \\
&\lesssim \lambda_n^{\frac{1}{p}}  \int_0^1 t^{\lambda_n}\Big(\frac{1}{1-t}\Big)^{\frac{1}{p'}}dt\\
&=  \lambda_n^{\frac{1}{p}} \int_{0}^{1-\frac{1}{\lambda_n}} t^{\lambda_n}\Big(\frac{1}{1-t}\Big)^{\frac{1}{p'}}dt +  \lambda_n^{\frac{1}{p}}\int_{1-\frac{1}{\lambda_n}}^1 t^{\lambda_n}\Big(\frac{1}{1-t}\Big)^{\frac{1}{p'}}dt \\
&\leq \lambda_n^{\frac{1}{p}}\lambda_n^{\frac{1}{p'}}\int_0^1t^{\lambda_n}dt + \lambda_n^{\frac{1}{p}}\int_{1-\frac{1}{\lambda_n}}^1 (1-t)^{-\frac{1}{p'}}dt\\
&\leq \frac{\lambda_n}{\lambda_n+1} + \lambda_n^{\frac{1}{p}} \frac{p}{\lambda_n^{\frac{1}{p}}}~~.
\end{align*}
We obtain that $D_n(p)$ is a bounded sequence of real numbers.
\end{proof}

From Prop.\ref{prop w mu Dn}, we obtain as claimed:
$$\Big\|\sum\limits_{n\in\mathbb{N}}b_nt^{\lambda_n}\Big\|_{p}\lesssim \Big(\sum\limits_{n\in\mathbb{N}}\frac{|b_n|^p}{\lambda_n}\Big)^{\frac{1}{p}}~~,$$
for any $b\in c_{00}$, when $\Lambda$ is lacunary.

Let us mention that from Lemma~\ref{lem 1 sur 1-t} and the Gurariy-Macaev's Theorem, one can easily get an estimation of the point evaluation on $M_\Lambda^p$:

\begin{prop}\label{prop evalponct}
Let $\Lambda$ be a quasi-geometric sequence and $p\ge1$. For any $t\in[0,1)$, the point evaluation $f\in M_\Lambda^p\longmapsto \delta_t(f)=f(t)$ satisfies
$$\big\|\delta_t\big\|_{(M_\Lambda^p)^\ast}\approx \frac{1}{(1-t)^\frac{1}{p}}\,\cdot$$

A fortiori, when $\Lambda$ is lacunary, we have $\dis\big\|\delta_t\big\|_{(M_\Lambda^p)^\ast}\lesssim \frac{1}{(1-t)^\frac{1}{p}}\,\cdot$
\end{prop}

\begin{proof} We fix $p>1$. Since $\Lambda$ is in particular lacunary, the Gurariy-Macaev theorem gives:
$$\big\|\delta_t\big\|_{(M_\Lambda^p)^\ast}=\sup_{f\in B_{M_\Lambda^p}}|f(t)|\approx\sup_{a\in B_{\ell^p}}\Big|\sum_{n\ge0}\lambda_n^\frac{1}{p}a_n t^{\lambda_n}\Big|=\Big(\sum_{n\ge0}\lambda_n^\frac{p'}{p} t^{p'\lambda_n}\Big)^\frac{1}{p'}$$
where the underlying constants depend on $p$ and $\Lambda$.
We conclude with Lemma~\ref{lem 1 sur 1-t}.

In the case $p=1$, we can easily adapt the argument, without using Lemma \ref{lem 1 sur 1-t}.
\end{proof}

\section{Revisiting the classical case}

In this section, we focus  mainly on the case $p>1$ and we shall consider the Lebesgue measure $\mu=m$ on $[0,1]$. We define the operator $$J_\Lambda:\left\lbrace\begin{array}{ccc}
\ell^p(\omega) & \longrightarrow & M_\Lambda^p\\
b&\longmapsto & \sum\limits_n b_nt^{\lambda_n}
\end{array}\right.~~$$
where the weights $\omega=(\omega_n)$ are given by
$\omega_n=(p\lambda_n+1)^{-1}=\|t^{\lambda_n}\|_p^p$. In particular, if we denote by $(e_k)_k$ the canonical basis of $\ell^p(\omega)$, we have $$\forall k\in\mathbb{N},~\|J_\Lambda(e_k)\|_{p}=\|e_k\|_{\ell^p(\omega)}~~.$$

The theorem of Gurariy-Macaev says that $J_\Lambda$ is an isomorphism if and only if $\Lambda$ is lacunary. Our Proposition \ref{prop gurariy rapide} proves as well that $J_\Lambda$ is bounded when $\Lambda$ is lacunary.

We are going to recover the boundedness of $J_\Lambda$ refining the method used for Prop.\ref{prop gurariy rapide}, in order to get a sharper estimate of the norm. Actually, we prove that $J_\Lambda$ is bounded if and only if $\Lambda$ is quasi-lacunary or $p=1$. Our approach is different from the one of Gurariy-Macaev (which was based on some slicing of the interval $(0,1)$), that is why we are able to control the constants of the norms with explicit quantities depending on the ratio of lacunarity (and $p$) only. As a consequence, we shall get that for $p\in(1,+\infty)$, $J_\Lambda$ is an asymptotical isometry if and only if $\Lambda$ is super-lacunary.

\begin{lem}\label{lem lac p p'}
Let $\alpha\in(0,+\infty)$, $p\in(1,+\infty)$ and $(q_n)_n$ be an $r$-lacunary sequence. We have $$\sup\limits_{n\in\mathbb{N}}\sum\limits_{\substack{k\in\mathbb{N}\\k\not=n}}\Bigg(\frac{q_n^{\frac{1}{p}}q_k^{\frac{1}{p'}}}{\displaystyle\frac{q_n}{p}+\frac{q_k}{p'}}\Bigg)^\alpha\leq \frac{p'^{\alpha}}{r^{\frac{\alpha}{p}}-1}+\frac{p^\alpha}{r^{\frac{\alpha}{p'}}-1}\cdot$$
\end{lem}
\begin{proof}
Let $n\in\mathbb{N}$. 
For $k< n$, we have $\displaystyle\frac{q_n^{\frac{1}{p}}q_k^{\frac{1}{p'}}}{\displaystyle\frac{q_n}{p}+\frac{q_k}{p'}}\leq p\Big(\frac{q_k}{q_n}\Big)^{\frac{1}{p'}}\leq pr^{-\frac{n-k}{p'}}\cdot$ We obtain:
$$\sum\limits_{k=0}^{n-1}\Big(\displaystyle\frac{q_n^{\frac{1}{p}}q_k^{\frac{1}{p'}}}{\displaystyle\frac{q_n}{p}+\frac{q_k}{p'}}\Big)^{\alpha}\leq p^\alpha\sum\limits_{k=0}^{n-1}\frac{1}{r^{\frac{(n-k)\alpha}{p'}}}\leq \frac{p^\alpha}{r^{\frac{\alpha}{p'}}-1}\cdot$$
When $k>n$, we have $\displaystyle\frac{q_n^{\frac{1}{p}}q_k^{\frac{1}{p'}}}{\displaystyle\frac{q_n}{p}+\frac{q_k}{p'}}\leq p'\Big(\frac{q_n}{q_k}\Big)^{\frac{1}{p}}\leq p'r^{-\frac{k-n}{p}}$
and, summing  over the $k$'s, we obtain the majorization.
\end{proof}

For $p\in[1,+\infty)$ we consider the sequence $D_n(p)$ defined in section 2: 
$$D_n(p)=\Bigg(\int_0^1 (p\lambda_n+1)^{\frac{1}{p}} t^{\lambda_n}\Big(\sum\limits_k (p\lambda_k+1)^{\frac{1}{p}} t^{\lambda_k}\Big)^{p-1}dt\Bigg)^{\frac{1}{p}}~~.$$

\begin{prop}\label{prop borne p>2}
Let  $p\geq 2$ and $\Lambda$ be a (lacunary) sequence such that $(p\lambda_n+1)_n$ is $r$-lacunary. Then we have:
$$\|J_{\Lambda}\|_p\leq \Bigg( 1+\frac{2p^{\frac{1}{p-1}}}{r^{\frac{1}{p(p-1)}}-1}\Bigg)^{\frac{1}{p'}}~~.$$
\end{prop}
\begin{proof}
For $j\in\mathbb{N}$, we denote $q_j=(p\lambda_j+1)$ and $f_j(t)=q_j^{\frac{1}{p}}t^{\lambda_j}=\displaystyle\frac{t^{\lambda_j}}{\|t^{\lambda_j}\|_p}\cdot$
We have:
\begin{align*}
D_n(p)^p=&\int_{0}^1f_n\Big(\sum\limits_k f_k\Big)^{p-1}dt=\Big\|\sum\limits_k f_k\Big\|_{L^{p-1}(f_ndt)}^{p-1}~~.
\end{align*}
Since $p-1\geq 1$, the triangle inequality gives:
\begin{align*}
D_n(p)^{p'}\leq \sum\limits_k \|f_k\|_{L^{p-1}(f_ndt)}
=\sum\limits_k\Big(  q_n^{\frac{1}{p}}q_k^{\frac{1}{p'}}\int_{0}^1t^{\lambda_n+(p-1)\lambda_k} dt \Big)^{\frac{1}{p-1}}.
\end{align*}
 For $n,k\in\mathbb{N}$, we have :
$$q_n^{\frac{1}{p}}q_k^{\frac{1}{p'}}\int_{0}^1t^{\lambda_n+(p-1)\lambda_k}dt=\frac{q_n^{\frac{1}{p}}q_k^{\frac{1}{p'}}}{\lambda_n+(p-1)\lambda_k+1}=\frac{q_n^{\frac{1}{p}}q_k^{\frac{1}{p'}}}{\displaystyle\frac{q_n}{p}+\frac{q_k}{p'}}\cdot$$
We apply Lemma~\ref{lem lac p p'} and we obtain for any $ n\in\mathbb{N}$:
$$D_n(p)^{p'}\leq \sum\limits_{k\in\mathbb{N}}\Big(\frac{q_n^{\frac{1}{p}}q_k^{\frac{1}{p'}}}{\displaystyle\frac{q_n}{p}+\frac{q_k}{p'}}\Big)^{\frac{1}{p-1}}\leq 1+\frac{2p^{\frac{1}{p-1}}}{r^{\frac{1}{p(p-1)}}-1}$$
since $p\geq p'$ and using that the term for $n=k$ is 1.
Thanks to  Prop.\ref{prop w mu Dn}, we have $$\|J_\Lambda\|_p=\|T_m\|_p\leq \sup\limits_nD_n(p).$$
\end{proof}

\begin{rem}
For $p\in(1,2)$, we can apply the same method and it would lead to:
$$\|J_\Lambda\|_p\leq \Bigg(1+\frac{2p'}{r^{\frac{1}{p'}}-1}\Bigg)^{\frac{1}{p}}~~.$$

But this bound is not sharp when $p$ is close to 1. For instance, it tends to $+\infty$ when $p\rightarrow 1$ and $r$ is fixed. But $\|J_\Lambda\|_1$ is always 1, without any assumption on $\Lambda$.\end{rem}

Point out that the operators $J_\Lambda:\ell^p(\omega)\to M_\Lambda^p\subset L^p(m)$ are not defined on the same scale of $L^p$-spaces, since the weight $\omega$ actually depends on $p$. We cannot apply directly Riesz-Thorin theorem for this problem, even not the weighted versions of the literature. Nevertheless, we shall adapt the proof in the next result and it gives the expected bound.

\begin{prop}\label{prop interpolation p<2}
Let $p\in[1,2]$ and let $\Lambda$ be a (lacunary) sequence such that $(p\lambda_n+1)_n$ is $r$-lacunary. Then we have:$$\|J_\Lambda\|_p\leq \Bigg(1+\frac{4}{r^{\frac{1}{2}}-1}\Bigg)^{\frac{1}{p'}}\cdot$$
\end{prop}
\begin{proof}
Our proof is adapted from the classical proof of Riesz-Thorin theorem, with an additional trick. 

Let $\theta=\displaystyle\frac{2}{p'}\in(0,1)$. We have $\displaystyle\frac{1}{p}=1-\frac{\theta}{2}\cdot$ As usual, for $z\in\mathbb{C}$ such that $0\leq Re(z)\leq 1,$ we define $\displaystyle\frac{1}{p(z)}=1-\frac{z}{2}$ and $\displaystyle\frac{1}{p'(z)}=\frac{z}{2}\cdot$ We have $p(\theta)=p$ and $p'(\theta)=p'$. We fix $a=(a_n)_n$ a sequence in $\mathbb{R}_+$ with a finite number of non-zero terms and  $g\in L^{p'}$ positive, such that 
$\|a\|_{\ell^p(\omega)}=\|g\|_{p'}=1.$ 
Finally we define $$F(z)=\sum_{n\in\mathbb{N}}a_n^{\frac{p}{p(z)}}\int_0^1t^{\frac{p}{p(z)}\lambda_n} g(t)^{\frac{p'}{p'(z)}}dt\;.$$
Point out that we actually have a finite sum, and $F$ is an holomorphic function on the band $\dis\{z\in\mathbb{C}|\; Re(z)\in(0,1)\}$.
For $x\in\mathbb{R},$ we have 
$$\displaystyle |F(ix)|\leq \sum\limits_{n\in\mathbb{N}}a_n^p\int_0^1t^{p\lambda_n}dt=\sum\limits_{n\in\mathbb{N}}\frac{a_n^p}{p\lambda_n+1}=1\;.$$ 

On the other hand, for every real number $x$:
\begin{align*}
|F(1+ix)|&\leq \sum\limits_{n\in\mathbb{N}} a_n^{p(1-\frac{1}{2})}\int_0^1t^{p(1-\frac{1}{2})\lambda_n}g(t)^{\frac{p'}{2}}dt\\
&=\int_0^1 g(t)^{\frac{p'}{2}}\sum\limits_{n\in\mathbb{N}}b_nt^{\psi_n}dt
\end{align*}
where $b_n=a_n^{\frac{p}{2}}$ and $\Psi=(\psi_n)_n=\Big(\displaystyle\frac{p\lambda_n}{2}\Big)_n$. Since $(2\psi_n+1)_n$ is also $r$-lacunary we can apply Prop.\ref{prop borne p>2}. in the hilbertian case:
$$\Big\|\sum\limits_{n\in\mathbb{N}}b_nt^{\psi_n}\Big\|_{2}=\|J_\Psi (b)\|_{2}\leq \Bigg(1+\frac{4}{r^{\frac{1}{2}}-1}\Bigg)^{\frac{1}{2}}\Bigg(\sum\limits_n\frac{|b_n|^2}{2\psi_n+1}\Bigg)^{\frac{1}{2}} \cdot$$
Since $\displaystyle\frac{1}{2\psi_n+1}=\frac{1}{p\lambda_n+1}$ and $|b_n|^{2}=|a_n|^p$, we have $$\sum\limits_n\frac{|b_n|^2}{2\psi_n+1}=\sum\limits_n\frac{|a_n|^p}{p\lambda_n+1}=1~~\cdot$$
We apply the Cauchy-Schwarz inequality and get:
$$|F(1+ix)|\leq \|g^{\frac{p'}{2}}\|_{2}\times \Big\|\sum\limits_{n}b_nt^{\psi_n}\Big\|_{2}\leq \Big(1+\frac{4}{r^{\frac{1}{2}}-1}\Big)^{\frac{1}{2}}~~.$$

Now, the proof finishes in a standard way and the three lines theorem gives
$$|F(\theta)|\le\Bigg(1+\frac{4}{r^{\frac{1}{2}}-1}\Bigg)^{\frac{\theta}{2}}\;.$$
From this, we conclude easily that for arbitrary $a\in\ell^p(\omega)$, we have
$$\|J_\Lambda(a)\|_{p}\leq \Big(1+\frac{4}{r^{\frac{1}{2}}-1}\Big)^{\frac{1}{p'}}\|a\|_{\ell^p(\omega)}~~.$$
\end{proof}

Now we can give a characterization of the boundedness of $J_\Lambda.$

\begin{thm}\label{thm quasi lac}
Let $p\in(1,+\infty)$.
The following are equivalent:
\begin{enumerate}[(i)]
\item  The sequence $\Lambda$ is quasi-lacunary ;
\item  The operator $J_\Lambda$ is bounded on $\ell^p(\omega)$.
\end{enumerate}
\end{thm}
\begin{proof}
Assume that $\Lambda$ is a quasi-lacunary sequence. Using Remark~\ref{rem qlac fini}, there exist $K\ge1$ and lacunary sets $\Lambda_j\subset\Lambda$ (with $j\in\{1,\cdots,K\}$) such that $\Lambda=\Lambda_1\cup\cdots\cup\Lambda_K$. 
We define the operators $$J^{(j)}:\left\lbrace\begin{array}{ccc}
\ell^p(\omega)&\longrightarrow &M_\Lambda^p \\
b&\longmapsto &\sum\limits_nb_nt^{\lambda_n}\ind_{\Lambda_j}(\lambda_n)
\end{array}\right.$$
where $\ind_{\Lambda_j}$ is the indicator function of the set $\Lambda_j$.

We have $J_\Lambda =\sum\limits_{j=1}^K J^{(j)}$. Moreover, for any $j$, the norm $\|J^{(j)}\|_p=\|J_{\Lambda_j}\|_p<+\infty$ thanks to Prop.\ref{prop interpolation p<2} and Prop.\ref{prop borne p>2}. Therefore, $J_{\Lambda}$ is bounded.

For the converse, we assume that $\Lambda$ is not quasi-lacunary. We denote $q_k=(p\lambda_k+1)$. For an arbitrarily large $N\in\mathbb{N}$ we consider the extraction $(Nk)_{k\in\mathbb{N}}$. It has bounded gaps, so the sequence $q_{Nk}$ is not lacunary. This implies $\liminf\limits_{k\rightarrow+\infty}\displaystyle\frac{q_{(k+1)N}}{q_{kN}}=1$, so there exists $k_0$ such that it is less than 2. For $n_0=k_0N$ we have $$q_{n_0+N}\leq 2q_{n_0}.$$

Let $A=\lbrace n_0,\ldots ,n_0+N-1 \rbrace$. Thanks to the inequality of arithmetic and geometric means, we have:
\begin{align*} 
\|J_\Lambda(\ind_A)\|_{p}^p=\displaystyle\int_0^1\Big|\sum\limits_{j\in A} t^{\lambda_j}\Big|^pdt
\geq \displaystyle\int_0^1 N^p\prod\limits_{j\in A}t^{\frac{p\lambda_j}{N}}dt.
\end{align*}
We obtain $$\|J_\Lambda(\ind_A)\|_{p}^p\geq \frac{N^p}{\sum\limits_{j\in A}\frac{q_j}{N}}\geq \frac{N^p}{q_{n_0+N}}\geq \frac{N^p}{2q_{n_0}}\cdot$$
On the other hand, $\|\ind_A\|_{\ell^p(\omega)}^p=\displaystyle\sum\limits_{j\in A}\frac{1}{q_j}\leq \frac{N}{q_{n_0}}\cdot$
Since $N$ is arbitrarily large and $p>1$, $J_\Lambda$ is not bounded.
\end{proof}

The following is a refinement of the Gurariy-Macaev theorem for the lacunary sequences with a large ratio.

\begin{thm}\label{thm r epsilon}
Let $p>1$. For any $\varepsilon\in(0,1)$, there exists $r_\varepsilon>1$ with the following property:

For any $\Lambda$ such that $(p\lambda_n+1)_n$ is $r_\varepsilon$-lacunary, we have:
$$\forall a\in\ell^p(\omega),\qquad(1-\varepsilon)\|a\|_{\ell^p(\omega)}\leq \|J_\Lambda(a)\|_{p}\leq (1+\varepsilon)\|a\|_{\ell^p(\omega)}~~.$$
\end{thm}

\begin{rem}
If we denote $q=\max\{p,p'\}$, the parameter $r_{\varepsilon}= \Big(1+\displaystyle\frac{4q^{\frac{1}{q-1}}}{\varepsilon}\Big)^{q(q-1)}$ is suitable for Th.\ref{thm r epsilon}.
\end{rem} 
\begin{proof}
Let $q=\max\{p,p'\}\geq 2$ and $r_{\varepsilon}= \displaystyle\Big(1+\displaystyle\frac{4q^{\frac{1}{q-1}}}{\varepsilon}\Big)^{q(q-1)}$. 

We fix a sequence $a\in\ell^p(\omega)$ with $\|a\|_{\ell^p(\omega)}=1$. Thanks to the choice of $r_\varepsilon$, when $p\geq 2$ we apply Prop.\ref{prop borne p>2} and we get that $\|J_\Lambda\|_p\leq \Big(1+\displaystyle\frac{\varepsilon}{2}\Big)^{\frac{1}{p'}}\leq 1+\frac{\varepsilon}{2}.$
When $p\leq 2$, Prop.\ref{prop interpolation p<2} gives also
$\displaystyle\|J_\Lambda\|_p\leq \Big(1+\frac{\varepsilon}{2})^{\frac{1}{p'}}\leq 1+\frac{\varepsilon}{2}\cdot$
In the two cases, the majorization part holds.

For the minoration part, we consider a sequence $b\in\ell^{p'}(\omega)$ such that $\|b\|_{\ell^{p'}(\omega)}=1.$ We define $\Psi=(\psi_n)_n$ by $\psi_n=\displaystyle\frac{p\lambda_n}{p'}=(p-1)\lambda_n.$ We have:
\begin{align*}
\|J_{\Lambda}(a) . J_{\Psi}(b)\|_{1}&=\int_0^1 \Big|\sum\limits_{n,k}a_nb_k t^{\lambda_n+(p-1)\lambda_k}\Big|dt\\
&\geq\Big|\sum\limits_{n=0}^{+\infty} \frac{a_nb_n}{p\lambda_n+1} \Big|-\sum_{\substack{n,k\in\mathbb{N}\\ k\not=n}}\frac{|a_n|.|b_k|}{\lambda_n+(p-1)\lambda_k+1}\cdot
\end{align*}
We introduce the sequence $(q_n)_n=(p\lambda_n+1)_n=(\omega_n^{-1})_n$. Since $\|a\|_{\ell^p(\omega)}=1$ and by duality we have $$\sup\Big\{\displaystyle\sum\limits_n\frac{a_nb_n}{p\lambda_n+1},\|b\|_{\ell^{p'}(\omega)}=1\Big\}=1.$$ We now majorize the second term.
For any $n,k$, Young's inequality gives: 
\begin{align*}
|a_nb_k|&=|a_n\omega_n^{\frac{1}{p}}b_k\omega_k^{\frac{1}{p'}}|\times q_n^{\frac{1}{p}}q_k^{\frac{1}{p'}}\\
&\leq \Big(\frac{1}{p}|a_n|^p\omega_n+\frac{1}{p'}|b_k|^{p'}\omega_k\Big)\times q_n^{\frac{1}{p}}q_k^{\frac{1}{p'}}~~.
\end{align*}
We sum over $n$ and $k$ we obtain:
\begin{align*}
\biindice{\sum}{n,k\in\mathbb{N}}{k\not=n}\frac{|a_n|.|b_k|}{\displaystyle\frac{q_n}{p}+\frac{q_k}{p'}}&\leq \frac{1}{p}\|a\|^p_{\ell^p(\omega)}\sup\limits_{n}\biindice{\sum}{k\in\mathbb{N}}{k\not=n}\frac{q_n^{\frac{1}{p}}q_k^{\frac{1}{p'}}}{\displaystyle\frac{q_n}{p}+\frac{q_k}{p'}}\\
&+\frac{1}{p'}\|b\|^{p'}_{\ell^{p'}(\omega)}\sup\limits_k\biindice{\sum}{n\in\mathbb{N}}{n\not=k}\frac{q_n^{\frac{1}{p}}q_k^{\frac{1}{p'}}}{\displaystyle\frac{q_n}{p}+\frac{q_k}{p'}}\cdot
\end{align*}
Applying Lemma~\ref{lem lac p p'}, this quantity is less than $\displaystyle\frac{2q}{r_\varepsilon^{\frac{1}{q}}-1}\leq \frac{\varepsilon}{2}$ thanks to the choice of $r_\varepsilon$ again.
On the other hand, Hölder's inequality gives
$$\|J_\Lambda (a).J_\Psi (b)\|_{L^1}\leq \|J_\Lambda (a)\|_{p}.\|J_\Psi (b)\|_{p'}\leq \Big(1+\frac{\varepsilon}{2}\Big)\|J_{\Lambda}(a)\|_{p}$$
because $(p'\psi_n+1)$ is also $r_\varepsilon$-lacunary, so we can apply the majorization part for  $\|J_\Psi\|_{p'}$.
Considering the upper bound over the sequences $b$, we finally obtain, for any $\Lambda$ at least $r_{\varepsilon}-$lacunary, and  for any $a$ in the unit sphere of $\ell^p(\omega)$,
$$ (1-\varepsilon)\leq \frac{1-\frac{1}{2}\varepsilon}{1+\frac{1}{2}\varepsilon}\leq\|J_{\Lambda}(a)\|_{p}\leq(1+\varepsilon)~~\cdot$$
\end{proof}

Before stating the next corollary, let us recall that a (normalized) sequence $(x_n)$ in a Banach space $X$ is asymptotically isometric to the canonical basis of $\ell^p$ if for every $\varepsilon\in(0,1)$, there exists an integer $N$ such that  
$$(1-\varepsilon)\Big(\sum\limits_{n\geq N}|a_n|^p\Big)^{\frac{1}{p}}\leq \Big\|\sum\limits_{n\geq N}a_nx_n\Big\|_X\leq (1+\varepsilon)\Big(\sum\limits_{n\geq N}|a_n|^p\Big)^{\frac{1}{p}}
$$
for any $a=(a_n)_n\in c_{00}$.

Equivalently there exists a null sequence $(\varepsilon_n)$ of positive numbers such that for every $N$, we have for any $a=(a_n)_n\in c_{00}$:
$$(1-\varepsilon_N)\Big(\sum\limits_{n\geq N}|a_n|^p\Big)^{\frac{1}{p}}\leq \Big\|\sum\limits_{n\geq N}a_nx_n\Big\|_{X}\leq (1+\varepsilon_N)\Big(\sum\limits_{n\geq N}|a_n|^p\Big)^{\frac{1}{p}}.
$$
When $p=2$, we can also say that such a sequence $(x_n)$ is asymptotically orthonormal.
\medskip

We can now prove
\begin{cor}\label{cor super lac}
Let $p\in(1,+\infty)$.
The following are equivalent:
\begin{itemize}
\item[$(i)$] $\Lambda$ is super-lacunary.
\item[$(ii)$] The sequence $\Big(\displaystyle\frac{t^{\lambda_n}}{\|t^{\lambda_n}\|_{p}}\Big)_n$ in $L^p$ is asymptotically isometric to the canonical basis of $\ell^p$.
\end{itemize}
\end{cor}
\begin{proof}
Assume that $\Lambda$ is super-lacunary: $\displaystyle\lim\limits_{n\rightarrow+\infty}\frac{\lambda_{n+1}}{\lambda_n}=+\infty$. As usual, we denote $q_n=(p\lambda_n+1)$, and $f_n(t)=q_n^{\frac{1}{p}}t^{\lambda_n}=\displaystyle\frac{t^{\lambda_n}}{\|t^{\lambda_n}\|_{p}}$.
We need to prove that for any $\varepsilon>0$, there exists $N\in\mathbb{N}$ such that
\begin{equation}\label{eq asymptotical iso}
(1-\varepsilon)\Big(\sum\limits_{n\geq N}|a_n|^p\Big)^{\frac{1}{p}}\leq \Big\|\sum\limits_{n\geq N}a_nf_n\Big\|_{p}\leq (1+\varepsilon)\Big(\sum\limits_{n\geq N}|a_n|^p\Big)^{\frac{1}{p}}
\end{equation}
for any $a=(a_n)_n\in c_{00}$. For a given $\varepsilon\in(0,1)$ we consider the number $r_\varepsilon$ given by Th.\ref{thm r epsilon}. Since $(q_n)_n$ is also super-lacunary, there is an integer $N$ large enough to insure that $q_{k+1}\geq r_\varepsilon q_k$ when $k\geq N$ and so the sequence $(p\lambda_{n+N}+1)_n$ is $r_\varepsilon-$lacunary. 
We apply the estimation of $\|J_\Lambda (\widetilde{a})\|_{p}$ given by Th.\ref{thm r epsilon} with the sequence $\widetilde{a}=\Big(a_nq_n^{\frac{1}{p}}\Big)_{n\ge N}$ and we get the result.

For the converse, let $\varepsilon\in(0,1)$. From the right hand inequality of $(\ref{eq asymptotical iso})$, we get the existence of an integer $N\in\mathbb{N}$ such that for any integer $n\geq N$, for any $u\in(0,1)$, 
$$\|f_n+uf_{n+1}\|_{p}\leq(1+\varepsilon)(1+u^p)^{\frac{1}{p}}\leq (1+\varepsilon)\Big(1+\frac{u^p}{p}\Big).$$
On the other hand, Hölder's inequality and $\|f_n^{p-1}\|_{p'}=1$ give
\begin{align*}
\|f_n+uf_{n+1}\|_{p}& \geq \int_0^1 (f_n+uf_{n+1})f_n^{p-1}dt\\
&=1+u\int_0^1f_{n+1}f_n^{p-1}dt
\end{align*}
We apply this for $u=\varepsilon^{\frac{1}{p}}$, we finally get:
$$\int_0^1f_{n+1}f_n^{p-1}dt\leq 3\varepsilon^{1-\frac{1}{p}}~~$$
and since $p>1$, we obtain $\displaystyle\int_0^1f_{n+1}f_n^{p-1}dt\rightarrow 0$ when $n\rightarrow+\infty.$ 

But
\begin{align*}
\int_0^1f_{n+1}f_n^{p-1}dt= \int_0^1 q_{n+1}^{\frac{1}{p}}q_n^{\frac{1}{p'}}t^{(p-1)\lambda_n+\lambda_{n+1}}dt\geq q_n \int_0^1t^{p\lambda_{n+1}}dt=\frac{q_n}{q_{n+1}}\cdot
\end{align*}
Thus, $\displaystyle\frac{p\lambda_n+1}{p\lambda_{n+1}+1}\rightarrow 0$ when $n\rightarrow+\infty$, and $\Lambda$ is super-lacunary. 
\end{proof}


\section{Carleson measures}

In this section, $\mu$ denotes a positive and finite measure on $[0,1)$ and $\Lambda$ is a fixed lacunary sequence. We shall generalize some results of \cite{CFT} and \cite{NT} with the estimations introduced in section 2. In particular, we give a positive answer to a question asked in \cite{NT}: if $\mu$ is a sublinear measure on $[0,1)$ and $\Lambda$ is lacunary, then the embedding operator $i_\mu^p:M_\Lambda^p\rightarrow L^p(\mu)$ is bounded.

\begin{defin}\label{def mesures}
Let $p\in[1,+\infty)$. We say that:
\begin{enumerate}[(i)]
\item $\mu$ is {\em sublinear} if there exists a constant $C>0$ such that $$\forall \eps\in(0,1),\qquad\mu([1-\varepsilon,1])\leq C \varepsilon~~;$$ 
The smallest admissible constant $C$ above is denoted $\|\mu\|_S$.
\item $\mu$ {\em satisfies $(B_p)$} when there exists a constant $C$ (depending only on $\Lambda$ and $p$) such that:
\begin{align}\tag{$B_p$}
\forall n\in\mathbb{N},\qquad\int_{[0,1)}t^{p\lambda_n}d\mu\leq \frac{C}{\lambda_n}\cdot
\end{align}
\item $\mu$ is a  {\em Carleson measure for $M_\Lambda^p$} when there exists a constant $C$ (depending only on $\Lambda$ and $p$) such that, for any M\"untz polynomial $f(t)=\sum\limits_n a_nt^{\lambda_n}$, 
$$\|f\|_{L^p(\mu)}\leq C \|f\|_p\;.$$
\end{enumerate}
\end{defin}

In this case we can define the following bounded embedding:
$$i_\mu^p:\left\lbrace\begin{array}{ccc}
M_\Lambda^p & \longrightarrow & L^p(\mu)\\
f& \longmapsto &f
\end{array}\right.~~.$$

\begin{rem}\label{rem mesures}
The notions defined above are connected to each other:
\begin{enumerate}[(i)]
\item If $\mu$ is a Carleson measure for $M_\Lambda^p$, then $\mu$ satisfies $(B_p)$.
\item For $p,q\in[1,+\infty)$ such that $p<q$, we have:
$$\mu\text{ is sublinear }\Rightarrow (B_p)\Rightarrow (B_q).$$
Indeed, since $t\in[0,1)\mapsto t^{p\lambda_n}$ is an increasing function, \cite[Lemma 2.2]{CFT} gives $$\displaystyle\int_{[0,1)}t^{p\lambda_n}d\mu\leq \|\mu\|_S\int_{0}^1t^{p\lambda_n}dt\leq \frac{p^{-1}\|\mu\|_S}{\lambda_n}\cdot$$
\item Moreover, if $\Lambda$ is a quasi-geometric sequence, and $\mu$ satisfies $(B_p)$ for some $p\in[1,+\infty)$ then $\mu$ is sublinear. It is essentially done in \cite{CFT} in the case $p=1$. More precisely, we have: $$\|\mu\|_S\leq 3pR\Big(\sup\limits_{n\in\mathbb{N}}\lambda_n\int_{[0,1)}t^{p\lambda_n}d\mu\Big)~~,$$
where $R$ is a constant such that $\lambda_{n+1}\leq R\lambda_n.$
\end{enumerate}
\end{rem}

The previous remarks suggest the natural question: does $(B_p)$ imply that $\mu$ is a Carleson measure for $M_\Lambda^p$ ? 

The answer is not clear in general. In \cite[Ex.6.2]{CFT}, they build a sublinear measure (so it satisfies $(B_1)$) and a sequence $\Lambda$ such that $\mu$ is not a Carleson measure for $M_\Lambda^1.$ But when $\Lambda$ is lacunary we shall see that the condition $(B_p)$ is almost sufficient for $\mu$ to be a Carleson measure for $M_\Lambda^p$, and even sufficient when $p=1$ or when $\Lambda$ is a quasi-geometric sequence.

The cornerstone of our approach is the following remark.
\begin{rem}\label{rem factorization}
For a lacunary sequence $\Lambda$, we can factorize $i_\mu^p$ through $\ell^p(w)$ as follows:
$$
\xymatrix{
M_\Lambda^p \ar[rr]^{i_\mu^p} \ar[rd]_{J_\Lambda^{-1}} &  & L^p(\mu) \\
 & \ell^p(w)\ar[ru]_{T_\mu} & 
}
$$
where $w=(w_n)_n$ is a weight satisfying $w_n\approx\lambda_n^{-1}$. With this kind of weight, the operator $J_\Lambda$ realizes an isomorphism between $\ell^p(w)$ and $M_\Lambda^p$: this is a rewording of the  Gurariy-Macaev Theorem (Th.\ref{thm gurariy}). $T_\mu$ is defined in section 2. The most natural weight is $w_n=(p\lambda_n+1)^{-1}$ but in this section, we are interested in estimations up to constants (possibly depending on $p$ and $\Lambda$). Of course, the results are the same with equivalent weights. So, we choose (in order to simplify) to fix the weight $w_n=\lambda_n^{-1}$.

In particular we obtain:
$$\|i_\mu^p\|\lesssim \|T_\mu\|_p\leq \sup\limits_{n}D_n(p)~~,$$
and for $n\in\mathbb{N}$ we have $$a_{n+1}(i_\mu^p)\lesssim a_{n+1}(T_\mu)\le D_n^\ast(p)$$
where the sequence $(D_n(p))_n$ is defined as in section 2 by the formula (here with our specified weight):
$$D_n(p)=\Big(\int_{[0,1)}\lambda_n^{\frac{1}{p}}t^{\lambda_n}\Big(\sum\limits_{k\in\mathbb{N}}\lambda_k^{\frac{1}{p}}t^{\lambda_k}\Big)^{p-1}d\mu\Big)^{\frac{1}{p}}~~.$$
\end{rem}

We first treat the case $p=1$.
\begin{prop}\label{prop i mu 1 borne}
Let $\Lambda=(\lambda_n)_n$ be a lacunary sequence. The following are equivalent:
\begin{enumerate}[(i)] 
\item $\mu$ satisfies $(B_1)$ ;
\item $\mu$ is a Carleson measure for $M_\Lambda^1$.
\end{enumerate}
In this case there exists a constant $C$ depending only on $\Lambda$ such that $$\|i_\mu^1\|\leq C\Big(\displaystyle\sup\limits_{n\in\mathbb{N}}\lambda_n\int_{[0,1)}t^{\lambda_n}d\mu\Big)\cdot$$
\end{prop}
\begin{proof}
$(ii)\Rightarrow(i)$ is obvious. For the converse, we apply the factorization described in Remark~\ref{rem factorization} : this gives $\|i_\mu^1\|\leq \|T_\mu\|_1.\|J_\Lambda^{-1}\|_1.$
On the other hand, Prop.\ref{prop an(T) Dn} gives $\|T_\mu\|_1\leq \sup\limits_{n}D_n(1)$ and we get the result. 
\end{proof}
As a corollary, we recover quickly \cite[Th.5.5]{CFT} in the lacunary case: the sublinear measures satisfy $(B_1)$, and so any sublinear measure is a Carleson measure for $M_\Lambda^1$.
For the general lacunary case, we have the following theorem.

\begin{thm}\label{thm Bp i mu q borne}
Let $\Lambda=(\lambda_n)_n$ be an $r$-lacunary sequence. Let $\mu$ be a positive measure on $[0,1)$ and $p\in[1,+\infty)$. We assume that $\mu$ satisfies $(B_p)$.

Then $\mu$ is a Carleson measure for $M_\Lambda^q$ for any $q>p$.
Moreover, we have $$\|i_\mu^q\|\leq C\Big(\displaystyle\sup\limits_{n\in\mathbb{N}}\lambda_n\int_{[0,1)}t^{p\lambda_n}d\mu\Big)^{\frac{1}{q}}$$ where $C$ depends only on $p,q$ and $\Lambda$.
\end{thm}
Before the proof,
we prove the following lemma.

\begin{lem}\label{lem Dn(q) borne}
Under the same assumptions of Th.\ref{thm Bp i mu q borne}, we have
$$D_n(q)^q\leq C \Big(\sup\limits_{k\geq n}\lambda_k\int_{[0,1)}t^{p\lambda_k}d\mu\Big)^{\frac{1}{p}}\Big(\sup\limits_{k\in\mathbb{N}}\lambda_k\int_{[0,1)}t^{p\lambda_k}d\mu\Big)^{\frac{1}{p'}}~~,$$
where $C$ is  constant depending only on $p,q$ and $r$.
\end{lem}
\begin{proof}
Since $(\lambda_k)_k$ is $r$-lacunary, for any $\beta\in\mathbb{R}_+^\ast$ we have:
$$\sum\limits_{k\leq n}\lambda_k^{\beta}\leq \frac{1}{1-r^{-\beta}}\lambda_n^\beta\text{ ~~ and ~~ }\sum\limits_{k>n}\lambda_k^{-\beta}\leq \frac{1}{r^{\beta}-1}\lambda_n^{-\beta}\cdot$$
For any $j\in\mathbb{N}$, we denote $M_j=\lambda_j\displaystyle\int_{[0,1)}t^{p\lambda_j}d\mu$ and $M=\sup\limits_j M_j<+\infty.$
Since $q>1$, we have for any $A,B\in\mathbb{R}_+$, $(A+B)^{q-1}\leq 2^{q-1}(A^{q-1}+B^{q-1})$. This gives:
\begin{align*}
D_n(q)^q&=\int_{[0,1)}\lambda_n^{\frac{1}{q}}t^{\lambda_n}\Big(\sum\limits_{k\in\mathbb{N}}\lambda_k^{\frac{1}{q}}t^{\lambda_k}\Big)^{q-1}d\mu\\
&\lesssim \int_{[0,1)}\lambda_n^{\frac{1}{q}}t^{\lambda_n}\Big(\sum\limits_{k\leq n}\lambda_k^{\frac{1}{q}}t^{\lambda_k}\Big)^{q-1}d\mu + \int_{[0,1)}\lambda_n^{\frac{1}{q}}t^{\lambda_n}\Big(\sum\limits_{k> n}\lambda_k^{\frac{1}{q}}t^{\lambda_k}\Big)^{q-1}d\mu
\end{align*}
We first majorize the first term above. If $p>1$, Hölder's inequality gives:
\begin{align*}
\int_{[0,1)}\lambda_n^{\frac{1}{q}}t^{\lambda_n}\big(\sum\limits_{k\leq n}\lambda_k^{\frac{1}{q}}t^{\lambda_k}\big)^{q-1}d\mu &\leq \lambda_n^{\frac{1}{q}}\Big(\int t^{p\lambda_n}d\mu\Big)^{\frac{1}{p}}\Big(\int\big(\sum\limits_{k\leq n}\lambda_k^{\frac{1}{q}}t^{\lambda_k}\big)^{p'(q-1)}d\mu\Big)^{\frac{1}{p'}}\\
&\leq M_n^{\frac{1}{p}}\lambda_n^{\frac{1}{q}-\frac{1}{p}}\Big(\sum\limits_{k\leq n}\lambda_k^{\frac{1}{q}}\|t^{\lambda_k}\|_{L^{p'(q-1)}(\mu)}\Big)^{q-1}
\end{align*}
where we used the triangle inequality since $p'(q-1)\geq p\geq 1$. For any $k\leq n$ we have $\displaystyle\int_{[0,1)}t^{p'(q-1)\lambda_k}d\mu\leq \int_{[0,1)}t^{p\lambda_k}d\mu\leq M_k\lambda_k^{-1}$. This gives:
\begin{align*}
\int_{[0,1)}\lambda_n^{\frac{1}{q}}t^{\lambda_n}\big(\sum\limits_{k\leq n}\lambda_k^{\frac{1}{q}}t^{\lambda_k}\big)^{q-1}d\mu &\leq \sup\limits_{k\leq n}M_k^{\frac{1}{p'}}M_n^{\frac{1}{p}}\lambda_n^{\frac{1}{q}-\frac{1}{p}}\Big(\sum\limits_{k\leq n}\lambda_{k}^{\frac{1}{q}-\frac{1}{p'(q-1)}}\Big)^{q-1} \\
&\lesssim \sup\limits_{k\leq n}M_k^{\frac{1}{p'}} M_n^{\frac{1}{p}}\lambda_n^{\frac{1}{q}-\frac{1}{p}} \lambda_n^{\frac{1}{q'}-\frac{1}{p'}}\\
&=\sup\limits_{k\leq n}M_k^{\frac{1}{p'}}M_n^{\frac{1}{p}}.
\end{align*}
If $p=1$, the inequality $t^{\lambda_k}\leq 1$ gives directly :
\begin{align*}
\int_{[0,1)}\lambda_n^{\frac{1}{q}}t^{\lambda_n}\big(\sum\limits_{k\leq n}\lambda_k^{\frac{1}{q}}t^{\lambda_k}\big)^{q-1}d\mu &\leq M_n\lambda_n^{-1}\lambda_n^{\frac{1}{q}}\Big(\sum\limits_{k\leq n}\lambda_k^{\frac{1}{q}}\Big)^{q-1}\lesssim M_n.
\end{align*}

For the second term we treat two cases. First if $q-1\geq p$, the triangle inequality gives:
\begin{align*}
\int_{[0,1)}\lambda_n^{\frac{1}{q}}t^{\lambda_n}\Big(\sum\limits_{k> n}\lambda_k^{\frac{1}{q}}t^{\lambda_k}\Big)^{q-1}d\mu &\leq \lambda_n^{\frac{1}{q}}\Big( \sum\limits_{k> n} \|\lambda_k^{\frac{1}{q}}t^{\lambda_k}\|_{L^{q-1}(t^{\lambda_n} \mu)}\Big)^{q-1}\\
&=\lambda_n^{\frac{1}{q}}\Big( \sum\limits_{k> n} \lambda_k^{\frac{1}{q}} \Big(\int_{[0,1)}t^{(q-1)\lambda_k+\lambda_n}d\mu \Big)^{\frac{1}{q-1}} \Big)^{q-1}\\
&\leq \lambda_n^{\frac{1}{q}}\Big( \sum\limits_{k> n} \lambda_k^{\frac{1}{q}} \Big(\int_{[0,1)}t^{p\lambda_k}d\mu \Big)^{\frac{1}{q-1}} \Big)^{q-1}\\
&\leq \sup\limits_{k>n} M_k \lambda_n^{\frac{1}{q}}\Big(\sum\limits_{k>n}\lambda_k^{\frac{1}{q}-\frac{1}{q-1}}\Big)^{q-1}\\
&\lesssim \sup\limits_{k>n}M_k\lambda_n^{\frac{1}{q}}\big(\lambda_n^{\frac{-1}{q(q-1)}}\big)^{q-1}=\sup\limits_{k>n}M_k.
\end{align*}
If $q-1<p$, let $\alpha=\displaystyle\frac{p}{p-(q-1)}\cdot$ It satisfies $\alpha>q$ and  $(q-1)\alpha'=p.$ We apply Hölder's inequality:
\begin{align*}
\int_{[0,1)}\lambda_n^{\frac{1}{q}}t^{\lambda_n}&\Big(\sum\limits_{k> n}\lambda_k^{\frac{1}{q}}t^{\lambda_k}\Big)^{q-1}d\mu \\
&\leq \lambda_n^{\frac{1}{q}}\Big(\int_{[0,1)} t^{\alpha\lambda_n}d\mu\Big)^{\frac{1}{\alpha}}\Big(\int_{[0,1)}\big(\sum\limits_{k> n}\lambda_k^{\frac{1}{q}}t^{\lambda_k}\big)^{p}d\mu\Big)^{\frac{1}{\alpha'}}\\
&\leq M_n^{\frac{1}{\alpha}}\lambda_n^{\frac{1}{q}-\frac{1}{\alpha}}\Big(\sum\limits_{k>n}\lambda_k^{\frac{1}{q}}\Big(\int_{[0,1)}t^{p\lambda_n}d\mu\Big)^{\frac{1}{p}}\Big)^{\frac{p}{\alpha'}}
\end{align*}
where we applied again the triangle inequality. We obtain:
\begin{align*}
\int_{[0,1)}\lambda_n^{\frac{1}{q}}t^{\lambda_n}\Big(\sum\limits_{k> n}\lambda_k^{\frac{1}{q}}t^{\lambda_k}\Big)^{q-1}d\mu &\leq M_n^{\frac{1}{\alpha}}\sup\limits_{k>n}M_k^{\frac{1}{\alpha'}}\lambda_n^{\frac{1}{q}-\frac{1}{\alpha}}\Big(\sum\limits_{k>n}\lambda_k^{\frac{1}{q}-\frac{1}{p}}\Big)^{q-1}\\
&\lesssim M_n^{\frac{1}{\alpha}}\sup\limits_{k>n}M_k^{\frac{1}{\alpha'}}.
\end{align*}
We finally get:
$$D_n(q)^q\lesssim M_n^{\frac{1}{p}} \sup\limits_{k\leq n}M_k^{\frac{1}{p'}} + \sup\limits_{k\geq n}M_k\cdot$$
\end{proof}

Now we can prove Th.\ref{thm Bp i mu q borne}.

\begin{proof}
Since $\Lambda$ is lacunary, we can factorize $i_\mu^q$ through $\ell^q(w)$ as in Remark~\ref{rem factorization}. We obtain $$\|i_\mu^q\|\lesssim \|T_\mu\|_q\leq \sup\limits_n D_n(q)$$
and Lemma~\ref{lem Dn(q) borne} gives the result.
\end{proof}

\begin{cor}
If $\mu$ is sublinear and $\Lambda$ is lacunary, then $\mu$ is a Carleson measure for $M_\Lambda^q$, for any $q\in[1,+\infty).$
\end{cor}
\begin{proof}
Remark~\ref{rem mesures} implies that the sublinear measures satisfy $(B_1)$, and we obtain:
$$\|i_\mu^q\|\lesssim \|\mu\|_S^{\frac{1}{q}}.$$
\end{proof}

The previous fact was proved for $p=2$ in \cite[Th.4.3]{NT}, and the authors announced the result for $p\in(1,2)$ (see \cite[Cor.5.2]{NT}). Unfortunately there is a gap in the proof of their interpolation result \cite[Th.5.1]{NT} : the interpolation is not easy to handle in M\"untz spaces because $f\in M_\Lambda^p$ does not imply that $|f|\in M_\Lambda^p$ in general.

Th.\ref{thm Bp i mu q borne} has the following interesting consequence.

\begin{cor}\label{cor Bp i mu q borne}
Let $\Lambda$ be a lacunary sequence and $p,q\in[1,+\infty)$ such that $p<q$. 
\begin{itemize}
\item[$(i)$] If $i_\mu^p$ is bounded, then $i_\mu^q$ is bounded.
\item[$(ii)$] The converse is false in general.
\end{itemize}
\end{cor}
\begin{proof}
If $i_\mu^p$ is bounded, then $\mu$ satisfies $(B_p)$. Th.\ref{thm Bp i mu q borne} imply that $i_\mu^q$ is bounded.
The point $(ii)$ is a consequence of the examples Ex.\ref{ex non borne compact} and Ex.\ref{ex Pascal} below.
\end{proof}

\begin{cor}\label{cor borne quasi geometric}
Let $q\in[1,+\infty)$ and let $\Lambda$ be a quasi-geometric sequence.  Then we have:
\begin{align*}
\|i_\mu^q\| &  \approx \sup\limits_{n}\Big(\int_{[0,1)}\lambda_nt^{q\lambda_n}d\mu\Big)^{\frac{1}{q}}  \approx \|\mu\|_S^{\frac{1}{q}} \\
&\approx \sup\limits_{n}\Big(\int_{[0,1)}\lambda_nt^{\lambda_n}d\mu\Big)^{\frac{1}{q}} \approx \sup\limits_nD_n(q)~~,
\end{align*}
where the underlying constants depend only on $q$ and $\Lambda.$

In particular, $\mu$ is a Carleson measure if and only if it is sublinear.
\end{cor}
\begin{proof}
Since $\Lambda$ is lacunary, Remark \ref{rem mesures} and Lemma \ref{lem Dn(q) borne} give easily:
$$\|i_\mu^q\|\lesssim \sup\limits_nD_n(q)\lesssim \sup\limits_n\Big(\lambda_n\int_{[0,1)}t^{\lambda_n}d\mu\Big)^{\frac{1}{q}}\lesssim \|\mu\|_S^{\frac{1}{q}}~~.$$

On the other hand, since $\Lambda$ quasi-geometric, Remark~\ref{rem mesures} $(iii)$ gives: $$\displaystyle\|\mu\|_S\lesssim\sup\limits_{n}\int_{[0,1)}\lambda_nt^{q\lambda_n}d\mu\leq \|i_\mu^q\|^q~~.$$
\end{proof}

\section{Compactness and Schatten classes}

In this part we are interested in the compactness of the embedding 
$$i_\mu^p:\left\lbrace\begin{array}{ccc}
M_\Lambda^p & \longrightarrow & L^p(\mu)\\
f& \longmapsto &f
\end{array}\right.$$
where $\mu$ is a Carleson measure for $M_\Lambda^p$.

We turn to the investigation of its membership to various classes of operator ideals. We are mainly interested in compactness and Schatten classes (when $p=2$).

As in section 4, we denote $w_n=\lambda_n^{-1}$; we consider the operators $J_\Lambda$ and $T_\mu$ and the sequence $D_n(p)$ associated to this weight.

\begin{defin}\label{def mesures compact}
Let $p\in[1,+\infty)$. We say that:
\begin{enumerate}[(i)]
\item $\mu$ is {\em vanishing sublinear} when $\displaystyle\lim\limits_{\varepsilon\rightarrow 0}\frac{\mu([1-\varepsilon,1])}{\varepsilon}=0~~;$ 
\item $\mu$ {\em satisfies $(b_p)$} when we have:
\begin{align}\tag{$b_p$}
\lim\limits_{n\rightarrow+\infty}\lambda_n\int_{[0,1)}t^{p\lambda_n}d\mu=0.
\end{align}
\end{enumerate}
\end{defin}

\begin{rem}\label{rem mesures compact}
Let $\mu$ be a Carleson measure for $M_\Lambda^p$. We have:
\begin{enumerate}[(i)]
\item if $i_\mu^p$ is compact and $p>1$, then $\mu$ satisfies $(b_p)$. 

To prove this, we remark that for any $k\in\mathbb{N}$ we have $$\displaystyle\int_0^1t^{\lambda_n}\lambda_n^{\frac{1}{p}}t^{k}dt=\frac{\lambda_n^{\frac{1}{p}}}{\lambda_n+k+1}\rightarrow 0~~\text{ when }n\rightarrow+\infty\cdot$$
Thus, for any polynomial $g$ we have $\displaystyle\int_0^1t^{\lambda_n}\lambda_n^{\frac{1}{p}}g(t)dt\rightarrow 0$. Since $p>1$ the polynomials are dense in $L^{p'}$ and so $(\lambda_n^{\frac{1}{p}}t^{\lambda_n})_n$ converges weakly to 0 in $M_\Lambda^p.$
The embedding $i_\mu^p$ is compact and so $\|\lambda_n^{\frac{1}{p}}t^{\lambda_n}\|_{L^p(\mu)}\rightarrow 0$ when $n\rightarrow+\infty.$
\item For $p,q\in[1,+\infty)$ such that $p<q$, we have:
$$\mu\text{ is vanishing sublinear }\Rightarrow (b_p)\Rightarrow (b_q).$$
Indeed, assume that $\mu$ is vanishing sublinear. For any $\varepsilon>0$, there exists $\eta>0$ such that $\|\mu_{|[1-\eta,1)}\|_S\leq \varepsilon$. We have :
$$\lambda_n\int_{[0,1)}t^{p\lambda_n}d\mu\leq \lambda_n\eta^{p\lambda_n}\mu([0,1))+\lambda_n\int_{[1-\eta,1)}t^{p\lambda_n}d\mu.$$
The first term tends to 0 when $n\rightarrow+\infty$ and the second is less than $p^{-1}\|\mu\vert_{[1-\eta,1)}\|_S\leq \frac{\varepsilon}{p}$ thanks to Remark~\ref{rem mesures}$(ii)$.
\item These assumptions are all equivalent to each other when $\Lambda$ is a quasi-geometric sequence. More precisely, for $\varepsilon>0$ close to 0, we have:
$$\frac{\mu([1-\varepsilon,1))}{\varepsilon}\leq 3pR\int_{[0,1)}t^{p\lambda_n}d\mu$$
where $n$ is the index such that $\displaystyle\varepsilon\in\Big(\frac{1}{p\lambda_{n+1}},\frac{1}{p\lambda_n}\Big]$, and $R$ is a constant such that $\lambda_{k+1}\leq R\lambda_k$ for any $k\in\mathbb{N}.$ We obtain that $\mu$ is vanishing sublinear in this case.
\end{enumerate}
\end{rem}
\medskip

\subsection{The case $p=1$.}~~\\

For $p=1$, when $i_\mu^1$ compact, $\mu$ still satisfies $(b_1)$ but the method to prove it is not the same as for $p>1$. 

\begin{prop}\label{prop p=1 b1 compact}
Let $\Lambda$ be a lacunary sequence.
The following are equivalent:
\begin{enumerate}[(i)]
\item $\mu$ satisfies $(b_1)$ ;
\item $i_\mu^1$ is compact ;
\item $i_\mu^1$ is weakly compact.
\end{enumerate}
\end{prop}

\begin{xrem}
Actually the implications $(ii)\Rightarrow(iii)\Rightarrow(i)$ are valid for any $L^1$-M\"untz space, without any assumption of lacunarity for $\Lambda$.

On another hand we can point out that, without any special assumption of lacunarity on $\Lambda$, the embedding $i_\mu^1$ is a Dunford-Pettis operator ({\em i.e.} maps a weakly convergent sequence into a norm-convergent sequence) if and only if $i_\mu^1$ is compact. This is due to the fact that $M_\Lambda^1$ has the Schur property since it is isomorphic to a subspace of $\ell^1$ (see \cite{We}, see also \cite{G} for some extensions of this result).
\end{xrem}
\begin{proof}
Let us prove that $(i)\Rightarrow(ii)$. Since $\Lambda$ is lacunary, we can factorize $i_\mu^1$ through $\ell^1(w)$ as in the proof of Th.\ref{thm Bp i mu q borne}: we have $i_\mu^1=T_\mu\circ J_\Lambda^{-1}$. On the other hand, $\mu$ satisfies $(b_1)$, so we have $D_n(1)=\lambda_n\displaystyle\int_{[0,1)}t^{\lambda_n}d\mu\rightarrow 0$ when $n\rightarrow +\infty.$ Prop.\ref{prop an(T) Dn} implies that $a_n(T_\mu)\rightarrow 0$ and we get $a_n(i_\mu^1)\rightarrow 0$ when $n\rightarrow+\infty.$

$(ii)\Rightarrow (iii)$ is obvious. 

$(iii)\Rightarrow(i)$. Assume now that $i_\mu^1$ is weakly compact. We denote $H=\lbrace\lambda_nt^{\lambda_n}\rbrace\subset L^1(\mu)$ and we fix $\eps\in(0,1)$. Since $H$ is bounded in $L^1(\mu)$ and weakly relatively compact, $H$ is uniformly integrable (see \cite[Th.III.C.12 p.137]{Wo}). This means that for any $\varepsilon>0$, there exists $\delta>0$ such that for any $n\in\mathbb{N}$ and any measurable set $A\subset[0,1)$ with $\mu(A)\leq \delta$, we have 
$$\int_A \lambda_nt^{\lambda_n}d\mu\leq \varepsilon.$$
Since $\mu(\lbrace 1\rbrace)=0$, there exists $s\in(0,1)$ such that $\mu([s,1))\leq \eta$. We have 
\begin{align*}
\int_{[0,1)}\lambda_n t^{\lambda_n}d\mu &=\int_{[0,s)}\lambda_n t^{\lambda_n}d\mu+\int_{[s,1]}\lambda_n t^{\lambda_n}d\mu\\
&\leq \lambda_ns^{\lambda_n}\mu([0,1))+\varepsilon.
\end{align*}
and since $\lambda_ns^{\lambda_n}\rightarrow 0$ when $n\rightarrow+\infty$ we obtain that $\mu$ satisfies $(b_1)$.
\end{proof}

\medskip

\subsection{The case $p>1$.}~~\\

Let us mention without proof  the next remark (the argument is the same as in Lemma \ref{lem schatten} below, but we shall not use this result in the general case).

\begin{rem} Let $\Lambda$ be a quasi-geometric sequence. There exist an integer $K\ge1$ and $C$ depending only on $\Lambda$ such that for any $n\in\mathbb{N}$ we have:
$$C\lambda_{n+K}\int_{[0,1)}t^{\lambda_{n+K}}d\mu \leq \lambda_n\int_{[0,1)}t^{p\lambda_n}d\mu\leq D_n(p)^p.$$
\end{rem}

We first give a first easy sufficient condition to ensure compactness. This is closely linked to the rough sufficient condition to ensure the boundedness of $i_\mu^p$ stated in Remark~\ref{rem bornebrut}

\begin{prop}\label{prop orderbounded}
Let $\Lambda$ be a quasi-geometric sequence. 

The Carleson embedding $i_\mu^p$ is order bounded if and only if $\dis\int_{[0,1)}\frac{d\mu}{ 1-t}\;dt<\infty. $
\end{prop} 
Point out that the previous integral condition is then sufficient to ensure that $i_\mu^p$ is a $p$-summing operator, hence compact from $M_\Lambda^p$ to $L^p(\mu)$.

\begin{proof}
Since the space $M_\Lambda^p$ is separable, $i_\mu^p$ is order bounded if and only if $\dis t\mapsto\sup_{f\in B_{M_\Lambda^p}}|f(t)|$ belongs to $L^p(\mu)$. Now, the estimation on the point evaluation (see Prop.\ref{prop evalponct}) gives the conclusion.
\end{proof}

In the same spirit than the boundedness problem, we can ``almost" characterize the compactness of $i_\mu^q$ for $q>1$, by testing the monomials.

\begin{thm}\label{thm bp i mu q compact}
Let $\Lambda$ be a lacunary sequence. Assume that $\mu$ satisfies $(b_p)$ for some $p\in[1,+\infty)$. Then $i_\mu^q$ is compact for any $q>p.$
\end{thm}
\begin{proof}
Since $\Lambda$ is lacunary, we can factorize $i_\mu^q$ through $\ell^q(w)$ as in Remark \ref{rem factorization}: $i_\mu^q=T_\mu\circ J_\Lambda^{-1}$ (recall that $J_\Lambda$ is an isomorphism). Prop.\ref{prop an(T) Dn} gives:
$$\|i_\mu^q\|_e\lesssim \|T_\mu\|_e\leq \limsup_{n\rightarrow+\infty}D_n(q)~~.$$
Since $\mu$ satisfies $(b_p)$, Lemma~\ref{lem Dn(q) borne} implies that $D_n(q)\rightarrow 0$ when $n\rightarrow+\infty$ and so $i_\mu^q$ is compact.
\end{proof}

\begin{cor}\label{cor bp i mu q compact}
Let $\Lambda$ be a lacunary sequence and $p,q\in[1,+\infty)$ such that $p<q$. 
\begin{enumerate}[(i)] 
\item If $i_\mu^p$ is compact, then $i_\mu^q$ is compact.
\item The converse is false in general.
\item If $\mu$ is vanishing sublinear, $i_\mu^p$ is compact.
\end{enumerate}
\end{cor}
\begin{proof}
If $i_\mu^p$ is compact, then $\mu$ satisfies $(b_p)$ and since $\Lambda$ is lacunary, Th.\ref{thm bp i mu q compact} gives that $i_\mu^q$ is compact.
The point $(ii)$ is a consequence of Example~\ref{ex non borne compact} or Example \ref{ex Pascal} below. At last $(iii)$ holds since any vanishing sublinear measure satisfies $(b_1).$
\end{proof}

\begin{cor}\label{cor compact quasi geometric}
Let $q\in[1,+\infty)$ and let $\Lambda$ be a quasi-geometric sequence. Assume that $\mu$ is a Carleson measure of $M_\Lambda^q$. Then we have:
\begin{align*}
\|i_\mu^q\|_e &  \approx \limsup\limits_{n}\Big(\int_{[0,1)}\lambda_nt^{\lambda_n}d\mu\Big)^{\frac{1}{q}}  \approx \Big(\limsup\limits_{\varepsilon\rightarrow 0}\frac{\mu([1-\varepsilon,1)}{\varepsilon}\Big)^{\frac{1}{q}}
 \approx \limsup\limits_{n\rightarrow+\infty}D_n(q)~~,
\end{align*}
where the underlying constants depend only on $q$ and $\Lambda.$

In particular, $i_\mu^q$ is compact if and only if $\mu$ is vanishing sublinear.
\end{cor}
\begin{proof}
We already saw in Lemma~\ref{lem Dn(q) borne} that:
$$\|i_\mu^q\|_e\lesssim\limsup\limits_{n\rightarrow+\infty}D_n(q)\lesssim\limsup\limits_{n}\Big(\int_{[0,1)}\lambda_nt^{\lambda_n}d\mu\Big)^{\frac{1}{q}}\lesssim \Big(\limsup\limits_{\varepsilon\rightarrow 0}\frac{\mu([1-\varepsilon,1)}{\varepsilon}\Big)^{\frac{1}{q}}~~,$$
this part only requires the lacunarity assumption on $\Lambda$.

To get the minoration of $\|i_\mu^q\|_e$ we use \cite[Th.3.5]{CFT} : they proved that
$$\|i_\mu^1\|_e=\lim\limits_{n\rightarrow +\infty}\|i_{\mu_n'}^1\|$$
where $\mu'_n$ is the restriction $\mu\vert_{[1-\frac{1}{n},1)}.$
The proof can be easily adapted for $q>1$ as it was noticed in \cite[Prop.2.6]{NT} and we have $$\|i_\mu^q\|_e=\lim\limits_{n\rightarrow+\infty}\|i_{\mu'_n}^q\|.$$
Since $\Lambda$ is quasi-geometric, Cor.\ref{cor borne quasi geometric} gives that there is a constant $C>0$ such that for any measure $\nu$ on $[0,1)$ we have: $\|i_\nu^q\|\geq C\|\nu\|_S^{\frac{1}{q}}.$ We have:
$$\|i_\mu^q\|_e=\lim\limits_{n\rightarrow+\infty}\|i_{\mu_n'}^q\|\geq C\lim\limits_{n\rightarrow+\infty}\|\mu\vert_{[1-\frac{1}{n},1)}\|_S^{\frac{1}{q}}=\Big(\limsup_{\varepsilon\rightarrow 0}\frac{\mu([1-\varepsilon,1)}{\varepsilon}\Big)^{\frac{1}{q}}\cdot$$
\end{proof}

The following result is an improvement of \cite[Prop.3.2]{CFT}. The result requires no assumption on the lacunarity of $\Lambda$ but a strong assumption on $\mu$.

\begin{prop}\label{prop nucleaire}
If  Supp($\mu$) is included in a compact set of $[0,1)$, then $i_\mu^p$ is a nuclear operator.
\end{prop}
\begin{proof}
Assume that $Supp(\mu)\subset [0,\delta]$ with $\delta<1$. We fix $\varepsilon>0$ such that $(1+\varepsilon)\delta <1$. 
Since $\Lambda$ satisfies the gap condition, we have the following classical estimation essentially done in 
\cite[Prop.6.2.2]{GL}:
there exists $K_\varepsilon$ such that for any M\"untz polynomial $f(t)=\sum\limits_ka_kt^{\lambda_k}$, we have $$|a_n|\leq K_\varepsilon(1+\varepsilon)^{\lambda_n}\|f\|_{p}~~.$$
This implies that the functionals $e_n^\ast:\left\lbrace \begin{array}{ccc}
M_\Lambda^p &\longrightarrow &\mathbb{C} \\
\sum\limits_k a_kt^{\lambda_k} & \longmapsto & a_n
\end{array}\right.$ are well defined, bounded, and we have $\|e_n^\ast\|\leq K_\varepsilon(1+\varepsilon)^{\lambda_n}~~.$

We define $g_n:[0,1)\rightarrow \mathbb{C}$ by $g_n(t)=t^{\lambda_n}.$ The functions $(g_n)_n$ belong to $L^p(\mu)$ and we have $\|g_n\|_{L^p(\mu)}\leq \mu([0,1))\delta^{\lambda_n}$. On the other hand,
for any M\"untz polynomial $f$, we have 
$i_\mu^pf=\sum\limits_{k}e_k(f)y_k$. So $i_\mu^p$ and $\sum\limits_{k\in\mathbb{N}}e_k^{\ast}\otimes y_k$ coincide on a dense set of $M_\Lambda^p$. 
Moreover, we have $$\sum\|e_k^{\ast}\otimes y_k\|\leq K_\varepsilon\mu([0,1))\sum\limits_k\Big(\delta(1+\varepsilon)\Big)^{\lambda_k}<+\infty$$
so $i_\mu^p$ is a nuclear operator.
\end{proof}
\medskip

\subsection{The case $p=2$.}~~\\ 

Now on we focus on the hilbertian framework.
\begin{lem}\label{lem schatten}
Let $\Lambda$ be a quasi-geometric sequence and $\mu$ such that $i_\mu^2$ is bounded.
\begin{enumerate}[(i)]
\item There exist an integer $K\ge1$ and $C>0$ depending only on $\Lambda$ such that for any $n\in\mathbb{N}$ we have:
$$C\lambda_{n+K}\int_{[0,1)}t^{\lambda_{n+K}}d\mu \leq \lambda_n\int_{[0,1)}t^{2\lambda_n}d\mu\leq D_n(2)^2.$$
\item For any $q\in(0,+\infty)$, we have:
$$\|(D_n(2))_n\|_{\ell^q}\approx \Big\|\Big(\displaystyle\lambda_n\int_{[0,1)}t^{2\lambda_n}d\mu\Big)^{\frac{1}{2}}_n\Big\|_{\ell^q}\approx \Big\|\Big(\displaystyle\lambda_n\int_{[0,1)}t^{\lambda_n}d\mu\Big)^{\frac{1}{2}}_n\Big\|_{\ell^q}$$
in the sense that these quantities are equivalent, up to constants depending only on $\Lambda$ and $q$.
\end{enumerate}
\end{lem}
\begin{proof}
For $n\in\mathbb{N}$ we have $$D_n(2)^2=\sum\limits_{k\in\mathbb{N}}(\lambda_n\lambda_k)^{\frac{1}{2}}\int_{[0,1)}t^{\lambda_n+\lambda_k}d\mu\geq \lambda_n\int_{[0,1)}t^{2\lambda_n}d\mu$$
since this last term is the term $n=k$ in the sum. On the other hand, we assume that $\Lambda$ is $r$-lacunary. There exists $K\in\mathbb{N}$ such that $r^K\geq 2$ and since $\Lambda$ is quasi-geometric, there exists $R\in\mathbb{R}$ such that $\lambda_{k+1}\leq R\lambda_k$ for any $k$. We obtain:
\begin{align*}
\lambda_{n+K}\int_{[0,1)}t^{\lambda_{n+K}}d\mu\leq R^K\lambda_n\int_{[0,1)}t^{r^K\lambda_n}d\mu\lesssim \lambda_n\int_{[0,1)}t^{2\lambda_n}d\mu
\end{align*}
and we obtain $(i)$.

For $k\in\mathbb{N}$ we shall denote $M_k=\displaystyle\lambda_k\int_{[0,1)}t^{\lambda_k}d\mu$. Assume that the sequence $(M_k^{\frac{1}{2}})_k\in\ell^q$. We compare $\|D_n(2)\|_{\ell^q} $ and $\|M_n^{\frac{1}{2}}\|$ and  shall, in some sense,  improve the estimation of Lemma~\ref{lem Dn(q) borne}.
For $n\in\mathbb{N},$ we have:
\begin{align*}
D_n(2)^2 &=\sum\limits_{k\leq n}(\lambda_n\lambda_k)^{\frac{1}{2}}\int_{[0,1)}t^{\lambda_n+\lambda_k}d\mu+\sum\limits_{k>n}(\lambda_n\lambda_k)^{\frac{1}{2}}\int_{[0,1)}t^{\lambda_n+\lambda_k}d\mu\\
&\leq \sum\limits_{k\leq n}(\lambda_n\lambda_k)^{\frac{1}{2}}\frac{M_n}{\lambda_n}+\sum\limits_{k>n}(\lambda_n\lambda_k)^{\frac{1}{2}}\frac{M_k}{\lambda_k}\\
&\leq M_n\frac{1}{1-\frac{1}{\sqrt{r}}}+\sum\limits_{k>n}M_k\frac{1}{\sqrt{r}^{k-n}}\cdot
\end{align*}
The number $D_n(2)^2$ is less than the $n$-th entry of the vector $A[(M_k)_k]$, where $A=(A_{n,k})_{n,k}$ is the matrix defined by 
$$A_{n,k}=\left\lbrace \begin{array}{ll}
0 &\text{ if } k<n\\
(1-r^{-\frac{1}{2}})^{-1} &\text{ if }k=n\\
\frac{1}{\sqrt{r}^{k-n}} &\text{ if }k>n.
\end{array}\right.
$$
Assume first that $q\geq 2$. 
Since $A$ satisfies $$\displaystyle\sup\limits_{n}\sum\limits_{k}A_{n,k}\leq \frac{2}{1-\frac{1}{\sqrt{r}}}~~\text{and}~~\displaystyle\sup\limits_{k}\sum\limits_{n}A_{n,k}\leq \frac{2}{1-\frac{1}{\sqrt{r}}},$$ we can apply the Schur lemma: $A$ defines a bounded operator $A:\ell^{\frac{q}{2}}\rightarrow\ell^{\frac{q}{2}}$ and we have $\displaystyle\|A\|_{\frac{q}{2}}\leq \frac{2}{1-\frac{1}{\sqrt{r}}}\cdot$ In particular, for $(M_k)_k\in \ell^{\frac{q}{2}}$ we obtain $$\|(D_n(2))\|_{\ell^q}\leq \frac{2}{1-\frac{1}{\sqrt{r}}}\|(M_k)^{\frac{1}{2}}\|_{\ell^q}.$$

Now we treat the case $q<2$. Since $\displaystyle\frac{q}{2}<1$, we have
\begin{align*}
D_n(2)^q &\leq \Big( M_n\frac{1}{1-\frac{1}{\sqrt{r}}}+\sum\limits_{k>n}M_k\frac{1}{\sqrt{r}^{k-n}}\Big)^{\frac{q}{2}}\\
&\leq \Big(M_n^{\frac{1}{2}}\Big)^q\frac{1}{(1-\frac{1}{\sqrt{r}})^\frac{q}{2}}+\sum\limits_{k>n}\Big(M_k^{\frac{1}{2}}\Big)^{q}\frac{1}{r^{\frac{q(k-n)}{4}}}
\end{align*}
And we get:
\begin{align*}
\sum\limits_{n\in\mathbb{N}}D_n(2)^q&\leq \sum\limits_n \Big(M_n^{\frac{1}{2}}\Big)^{q}\Big(\frac{1}{1-\frac{1}{\sqrt{r}}}\Big)^{\frac{q}{2}}+ \sum\limits_{k\in\mathbb{N}}\Big(M_k^{\frac{1}{2}}\Big)^q\sum\limits_{n=0}^{k-1}\Big(\frac{1}{r}\Big)^{\frac{q(k-n)}{4}}\\
&\lesssim\|M_n^{\frac{1}{2}}\|_{\ell^q}^q
\end{align*}
where the underlying constants depend on $r$ and $q$ only. 
\end{proof}

\begin{thm}
Let $\Lambda$ be a lacunary sequence and $q>0$. We have
\begin{enumerate}[(i)] 
\item If  $(D_n(2))_n\in\ell^q$ then we have:$$\|i_\mu^2\|_{\mathcal{S}^q}\lesssim \|D_n(2)\|_{\ell^q}~~;$$
\item If moreover we assume that $\Lambda$ is quasi-geometric, and $q\geq 2$, then we have:
$$\|i_\mu^2\|_{\mathcal{S}^q}\approx \|D_n(2)\|_{\ell^q}~~,$$
where the underlying constants depend only on $q$ and $\Lambda$.
\end{enumerate}
\end{thm}
\begin{proof}
As in Remark~\ref{rem factorization}, since $\Lambda$ is lacunary we can factorize $i_\mu^2$ through $\ell^2(w)$, and we get $a_n(i_\mu^2)\lesssim a_n(T_\mu)$ and  Prop.\ref{prop an(T) Dn} gives $$\sum\limits_{n}(a_n(i_\mu^2))^q\lesssim \sum\limits_{n}D_n(2)^q~~.$$ 

Assume now that $q\geq 2$. As a direct consequence of \cite[Th.4.7 p.82]{DJT}, we obtain that for any Riesz basis $(f_n)_n$ of $M_\Lambda^2$, there exists a constant $C>0$ such that $$\|i_\mu^2\|_{\mathcal{S}^q}\geq C \Big(\sum\limits_n\|f_n\|_{L^2(\mu)}^q\Big)^{\frac{1}{q}}~~.$$

The theorem of Gurariy-Macaev says exactly that the sequence $(f_n)_n=(\lambda_n^{\frac{1}{2}}t^{\lambda_n})_n$ is a Riesz basis of $M_\Lambda^2$, and we obtain:
\begin{align*}
\|i_\mu^2\|_{\mathcal{S}^q}^q &\gtrsim \sum\limits_n\Big(\lambda_n\int_{[0,1)}t^{2\lambda_n}d\mu\Big)^{\frac{q}{2}}
\end{align*}
and Lemma~\ref{lem schatten} gives the result.
\end{proof}

We also have an integral expression for $\|i_\mu^2\|_{\mathcal{S}^q}$.

\begin{prop}\label{prop integrales}
Assume that $\Lambda$ is quasi-geometric and $q\geq 2.$ We have:
$$\|i_\mu^2\|_{\mathcal{S}^q}\approx \Big(\int_0^1 \Big(\int_{[0,1)}\frac{d\mu(t)}{(1-st)^{\frac{2}{q}+1}}\Big)^{\frac{q}{2}}ds\Big)^{\frac{1}{q}}~~. $$
\end{prop}
 \begin{proof}
We denote $M_n=\lambda_n\displaystyle\int_{[0,1)}t^{2\lambda_n}d\mu$. The previous estimation gives:  
\begin{align*}
\|i_\mu^2\|_{\mathcal{S}^q} \approx \Big(\sum\limits_nM_n^{\frac{q}{2}}\Big)^{\frac{1}{q}}
=\|(M_n)_n\|_{\ell^{\frac{q}{2}}}^{\frac{1}{2}}~~.
\end{align*}
On the other hand we can apply the theorem of Gurariy-Macaev to estimate an equivalent of $\|(M_n)\|_{\ell^{\frac{q}{2}}}$. We obtain, using Lemma~\ref{lem 1 sur 1-t},

\begin{align*}
\|i_\mu^2\|_{\mathcal{S}^q} &\approx\Big\|\sum\limits_{n}M_n\lambda_n^{\frac{2}{q}}s^{\lambda_n}\Big\|_{L^{\frac{q}{2}}(ds)}^{\frac{1}{2}}
=\Big(\int_0^1 \Big(\sum\limits_{n}\lambda_n\displaystyle\int_{[0,1)}t^{2\lambda_n}d\mu(t)\lambda_n^{\frac{2}{q}}s^{\lambda_n}\Big)^{\frac{q}{2}}ds \Big)^{\frac{1}{q}}
\\
&=\Big(\int_0^1 \Big(\int_{[0,1)}\sum\limits_{n}\lambda_n^{\frac{2}{q}+1}(st^2)^{\lambda_n}d\mu(t)\Big)^{\frac{q}{2}}ds\Big)^{\frac{1}{q}}\\
&\approx\Big(\int_0^1 \Big(\int_{[0,1)}\frac{d\mu(t)}{(1-st^2)^{\frac{2}{q}+1}}\Big)^{\frac{q}{2}}ds\Big)^{\frac{1}{q}}~~.
\end{align*} 
We get the result since $(1-st)\leq (1-st^2)\leq(1+st)(1-st)\le 2(1-st)$ for $s,t\in[0,1].$
 \end{proof}

Note that the previous criterion is the same for any sequence $\Lambda$ which is quasi-geometric.
In particular, we have a characterization of the Hilbert-Schmidt embeddings.

\begin{thm}\label{thm Hilbert Schmidt}
Let $\Lambda$ be a quasi-geometric sequence. The following are equivalent:
\begin{enumerate}[(i)] 
\item $i_\mu^2$ is an Hibert-Schmidt operator ;
\item $\displaystyle\int_{[0,1)}\frac{1}{1-t}d\mu <+\infty$ ;
\end{enumerate}
In this case we have $\dis\|i_\mu^2\|_{\mathcal{S}^2}\approx \Big(\displaystyle\int_{[0,1)}\frac{1}{1-t}d\mu\Big)^{\frac{1}{2}}~~.$
\end{thm}
\begin{proof}
Proof 1. We apply Prop.\ref{prop integrales} in the case $q=2$.
The Fubini theorem gives:
\begin{align*}
\|i_\mu^2\|_{\mathcal{S}^2}^2&\approx\int_0^1 \int_{t\in[0,1)}\frac{d\mu(t)}{(1-st)^{2}}ds= \int_{[0,1)}\frac{1}{1-t}d\mu ~~.
\end{align*}

Proof 2. It suffices to invoke the fact that order bounded and Hilbert-Schmidt operators are the same in an $L^2$-framework, and Prop.\ref{prop orderbounded} gives the result.
\end{proof}
\medskip

\subsection{Examples}~~\\

Now we give two examples, showing that in a strong manner, the boundedness and the compactness of Carleson embeddings on M\"untz spaces $M_\Lambda^p$ depend  in general on $p$ and not only on $\Lambda$.

\begin{exa}\label{ex non borne compact}
Let $p\in[1,+\infty).$
We are going to construct a lacunary sequence $\Lambda$ and a measure $\mu$ on $[0,1)$ such that 
\begin{enumerate}[(A)] 
\item $i_\mu^q$ is not bounded when $q\in[1,p]$ ;
\item $i_\mu^q$ is compact when $q\in(p,+\infty)$.
\end{enumerate}
\end{exa}

\begin{proof}
Note that $\Lambda$ cannot be a quasi-geometric sequence.
We shall take a measure $\mu$ with the form $\mu=\sum\limits_{k\geq 2}c_k\delta_{x_k}$ where $x_k\in(0,1)$ and $c_k>0$. 

We define $\lambda_2=1$, $(\lambda_n)_{n\geq 2}$ such that for any $n\geq 3$, we have $\lambda_{n}\geq n^{p+1}\lambda_{n-1}$.
For $n\geq 2$ let $c_n=\displaystyle\frac{n^p\log(n)}{\lambda_n}$ and $\displaystyle x_n=1-\frac{\log(n)}{\lambda_n}\cdot$ We have $\displaystyle x_n^{\lambda_n}\sim \frac{1}{n}$ when $n\rightarrow+\infty$, and in for $n,k$ such that $n\geq k$ we have $x_k^{\lambda_n}\lesssim \Big(\displaystyle\frac{1}{k}\Big)^{\frac{\lambda_n}{\lambda_k}}~~.$
We check that $\mu$ does not satisfy $(B_p)$:
\begin{align*}
\lambda_n\int_{[0,1)}t^{p\lambda_n}d\mu &=\sum\limits_k\lambda_nc_kx_k^{p\lambda_n}
\geq \lambda_nc_nx_n^{p\lambda_n}\sim \lambda_n\frac{n^p\log(n)}{\lambda_n}\frac{1}{n^p}=\log(n)\rightarrow +\infty.
\end{align*}
Hence $i_\mu^p$ is not bounded.

On the other hand, for $q>p$, we have 
$$\lambda_n\int_{[0,1)}t^{q\lambda_n}d\mu=\sum\limits_{k<n}\lambda_n c_k x_k^{q\lambda_n}+\lambda_nc_nx_n^{q\lambda_n}+\sum\limits_{k>n}\lambda_n c_k x_k^{q\lambda_n}.$$
We control these three terms. For the first:
$$\sum\limits_{k<n}\lambda_n c_k x_k^{q\lambda_n}\lesssim \sum\limits_{k<n}\log(k)k^p\frac{\lambda_n}{\lambda_k}\Big(\frac{1}{k^q}\Big)^{\frac{\lambda_n}{\lambda_k}}\lesssim \sum\limits_{k<n}\frac{\lambda_n}{\lambda_k}\Big(\frac{1}{k^q}\Big)^{\frac{\lambda_n}{\lambda_k}-1}\cdot$$
Since $k\geq 2$ and $\displaystyle\frac{\lambda_n}{\lambda_{n-1}}\rightarrow+\infty$, this term tends to 0 when $n\rightarrow+\infty.$
For the term $n=k$ we have : $\displaystyle\lambda_nc_nx_n^{q\lambda_n}\sim\lambda_n\frac{n^p\log(n)}{\lambda_n}\frac{1}{n^q}=\frac{\log(n)}{n^{q-p}}\rightarrow 0.$
For the last sum, $x_k\leq 1$ gives:
\begin{align*}
\sum\limits_{k>n}\lambda_n x_k^{q\lambda_n}c_k &\leq \sum\limits_{k=n+1}^{+\infty}\lambda_{n}\frac{k^p\log(k)}{\lambda_{k}}
\leq \sum\limits_{k=n+1}^{+\infty} \frac{\log(k)}{k}\times\frac{\lambda_n}{\lambda_{k-1}}\\
&\lesssim \frac{\log(n)}{n}\sum\limits_{k=n}^{+\infty}\frac{\lambda_n}{\lambda_k}\rightarrow 0.
\end{align*}
Thus, $\mu$ satisfies $(b_q)$, and using Th.\ref{thm bp i mu q compact} $i_\mu^{r}$ is compact for any $r>q$. We obtain that for any $r>p$, $i_\mu^r$ is compact.
\end{proof}

\begin{exa}\label{ex Pascal}
Let $p\in(1,+\infty).$
We shall construct a lacunary sequence $\Lambda$ and a measure $\mu$ on $[0,1)$ such that 
\begin{enumerate}[(A)]
\item $i_\mu^q$ is not bounded when $q\in[1,p)$ ;
\item $i_\mu^q$ is compact when $q\in[p,+\infty)$.
\end{enumerate}
\end{exa}

\begin{proof}
We take again a measure $\mu$ with the form $\mu=\sum\limits_{k\geq 2}c_k\delta_{x_k}$. Let  $\Lambda=(\lambda_n)_{n\geq 2}$ with $\lambda_2=1$, 
and for all $n\geq 3$, $\lambda_n\geq n^{p\max\{p,p'\}}\lambda_{n-1}.$

Let $c_n=\displaystyle\frac{n^p}{\lambda_n\log(n)}$ and $\displaystyle x_n=1-\frac{\log(n)}{\lambda_n}\cdot$ We have $\displaystyle x_n^{\lambda_n}\sim \frac{1}{n}$ when $n\rightarrow+\infty$, and in for $n,k$ such that $n\geq k$ we have $x_k^{\lambda_n}\lesssim \Big(\displaystyle\frac{1}{k}\Big)^{\frac{\lambda_n}{\lambda_k}}~~.$

Let $q\in[1,p)$. 
We check that $\mu$ does not satisfy $(B_q)$:
\begin{align*}
\lambda_n\int_{[0,1)}t^{q\lambda_n}d\mu &
\geq \lambda_nc_nx_n^{q\lambda_n}\sim \lambda_n\frac{n^p}{\lambda_n\log(n)}\frac{1}{n^q}=\frac{n^{p-q}}{\log(n)}\rightarrow +\infty.
\end{align*}
Hence $i_\mu^q$ is not bounded.
On the other hand, we show that the sequence $D_n(p)$ tends to 0 when $n\rightarrow+\infty$:
\begin{align*}
D_n(p)^p&=\sum\limits_{j\in\mathbb{N}}\lambda_n^{\frac{1}{p}}c_jx_j^{\lambda_n}\Big(\sum\limits_k \lambda_k^{\frac{1}{p}} x_j^{\lambda_k}\Big)^{p-1}\\
&\lesssim \lambda_n^{\frac{1}{p}}c_nx_n^{\lambda_n}\Big(\sum\limits_k\lambda_k^{\frac{1}{p}}x_n^{\lambda_k}\Big)^{p-1} +\sum\limits_{j\not=n}\lambda_n^{\frac{1}{p}}c_jx_j^{\lambda_n}\Big(\frac{1}{1-x_j}\Big)^{\frac{1}{p'}} 
\end{align*}
using Lemma~\ref{lem 1 sur 1-t} and Remark~\ref{rem 1 sur 1-t} for the second term. We first control the second term. If $j>n$, $x_j^{\lambda_n}\leq 1$ gives: 
$$\sum\limits_{j>n}\lambda_n^{\frac{1}{p}}c_jx_j^{\lambda_n}\Big(\frac{1}{1-x_j}\Big)^{\frac{1}{p'}} \leq \sum\limits_{j>n}\lambda_n^{\frac{1}{p}}\frac{j^p}{\lambda_j}\frac{\lambda_j^{\frac{1}{p'}}}{\log(j)^{1+\frac{1}{p'}}}\leq \sum\limits_{j>n}j^p\Big(\frac{\lambda_n}{\lambda_j}\Big)^{\frac{1}{p}}\leq \sum\limits_{j>n}\frac{1}{j^p}$$
since $\lambda_j\geq j^{p^2}\lambda_{j-1}.$ Hence this term tends to 0.

For $j<n$ we have $\displaystyle x_j^{\lambda_n}\lesssim \Big(\frac{1}{j}\Big)^{\frac{\lambda_n}{\lambda_j}}$ and we obtain:
$$\sum\limits_{j<n}\lambda_n^{\frac{1}{p}}c_jx_j^{\lambda_n}\Big(\frac{1}{1-x_j}\Big)^{\frac{1}{p'}} \lesssim \sum\limits_{j<n}\lambda_n^{\frac{1}{p}}\frac{j^p}{\lambda_j}\Big(\frac{1}{j}\Big)^{\frac{\lambda_n}{\lambda_j}}\frac{\lambda_j^{\frac{1}{p'}}}{\log(j)^{1+\frac{1}{p'}}}\leq \sum\limits_{j<n}\Big(\frac{\lambda_n}{\lambda_j}\Big)^{\frac{1}{p}}\Big(\frac{1}{j}\Big)^{\frac{\lambda_n}{\lambda_j}-p}$$
and since $j\geq 2$ and $\displaystyle\frac{\lambda_n}{\lambda_{n-1}}\rightarrow +\infty$, this term tends to 0 when $n\rightarrow+\infty.$

To majorize the part "$j=n$" we split the sum in three terms:
\begin{align*}
\lambda_n^{\frac{1}{p}}c_nx_n^{\lambda_n}\Big(\sum\limits_k\lambda_k^{\frac{1}{p}}x_n^{\lambda_k}\Big)^{p-1} &\lesssim \lambda_n^{\frac{1}{p}}c_nx_n^{\lambda_n}\Big(\sum\limits_{k<n}\lambda_k^{\frac{1}{p}}x_n^{\lambda_k}\Big)^{p-1} + \lambda_nx_n^{p\lambda_n}c_n\\
&+\lambda_n^{\frac{1}{p}}c_nx_n^{\lambda_n}\Big(\sum\limits_{k>n}\lambda_k^{\frac{1}{p}}x_n^{\lambda_k}\Big)^{p-1}
\end{align*}

For $k<n$, we have $x_n\leq 1$ and it gives:
\begin{align*}
\lambda_n^{\frac{1}{p}}c_nx_n^{\lambda_n}\Big(\sum\limits_{k<n}\lambda_k^{\frac{1}{p}}x_n^{\lambda_k}\Big)^{p-1}\lesssim \frac{\lambda_n^{\frac{1}{p}}n^p}{\log(n)\lambda_n} \frac{1}{n} \Big(\sum\limits_{k\leq n-1}\lambda_k^{\frac{1}{p}}\Big)^{p-1}\lesssim n^{p-1} \Big(\frac{\lambda_{n-1}}{\lambda_n}\Big)^{\frac{1}{p'}}\leq \frac{1}{n}
\end{align*}
since $\lambda_n\geq \lambda_{n-1}n^{pp'}.$

For the term $n=k$, we have $\lambda_nx_n^{p\lambda_n}c_n\sim \displaystyle\frac{\lambda_nn^p}{n^p\lambda_n\log(n)}=\frac{1}{\log(n)}\rightarrow 0.$ For $k> n$, we have $x_n^{\lambda_k}\lesssim \Big(\displaystyle\frac{1}{n}\Big)^{\frac{\lambda_k}{\lambda_n}}$ and we obtain:

\begin{align*}
\lambda_n^{\frac{1}{p}}c_nx_n^{\lambda_n}\Big(\sum\limits_{k>n}\lambda_k^{\frac{1}{p}}x_n^{\lambda_k}\Big)^{p-1} &\lesssim \frac{n^{p-1}}{\log(n)}\lambda_n^{-\frac{1}{p'}}\Big(\sum\limits_{k>n}\lambda_k^{\frac{1}{p}}\Big(\frac{1}{n}\Big)^{\frac{\lambda_k}{\lambda_n}}\Big)^{p-1}\\
&\leq \Big(\sum\limits_{k>n}\Big(\frac{\lambda_k}{\lambda_n}\Big)^{\frac{1}{p}}\Big(\frac{1}{n}\Big)^{\frac{\lambda_k}{\lambda_n}-1}\Big)^{p-1}
\end{align*}
and this term tends to 0 since $\displaystyle\frac{\lambda_{n+1}}{\lambda_{n}}\rightarrow+\infty.$

Thus, $D_n(p)\rightarrow 0$ when $n\rightarrow+\infty$. Since $\Lambda$ is lacunary we can factorize $i_\mu^p$ as in Remark \ref{rem factorization}. We have $i_\mu^p=T_\mu\circ J_\Lambda^{-1}$ (recall that $J_\Lambda$ is an isomorphism) and $T_\mu$ is compact thanks to Prop.\ref{prop an(T) Dn}. Hence Cor.\ref{cor bp i mu q compact} implies that $i_\mu^{q}$ is compact for any $q\geq p$.
\end{proof}
\bibliographystyle{plain}

\begin{thebibliography}{99}

\bibitem[AHLM]{AHLM}{I. AlAlam and  G.Habib and  P. Lef\`evre and  F. Maalouf,} {\sl Essential norms of Volterra and Ces\`aro operators on  M\"untz Spaces, }{Colloquium  Math.(to appear).}

\bibitem[AL]{AL}{I. AlAlam and  P. Lef\`evre, }{\sl Essential norms of weighted composition operators on $ L^1$  M\"untz spaces, }{Serdica Math. J., no.40(3) (2014), p.241-260.}


\bibitem[BE]{BE}{P. Borwein and T. Erdelyi, }{\sl Polynomials and polynomial inequalities, }{Springer, 1995.}

\bibitem[CFT]{CFT}{I. Chalendar and E. Fricain and D. Timotin, }{\sl Embeddings theorems for M\"untz Spaces, }{Ann. Inst. Fourier (Grenoble), no.6 (61), (2011) MR2976312, 2291-2311.}


\bibitem[DJT]{DJT}{J. Diestel and H. Jarchow and A. Tonge, }{\sl Absolutely summing operators, }{Cambridge Univesity press, 1995.}

\bibitem[G]{G}{G. Godefroy, }{\sl Unconditionality in spaces of smooth functions, }{Arch. Math. (Basel) 92 (2009), no. 5, 476-484.}

\bibitem[GL]{GL}{V. Gurariy and W. Lusky, }{\sl Geometry of M\"untz spaces and related questions, }{Springer-Verlag, Berlin, 2005.}


\bibitem[GM]{GM}{V. Gurariy and V.I. Macaev, }{\sl Lacunary power sequences in the spaces $C$ and $L^p$, }{Amer. Math. Soc. Translated, Serie 2, Vol.72, (1966), 9-21.}


\bibitem[LL]{LL}{S.V.  Ludkovsky and W. Lusky, } {\sl On the geometry of M\"untz spaces, }{ J. Funct. Spaces, (2015) Art. ID 787291, 7 pp.}


\bibitem[NT]{NT}{W.S. Noor and D. Timotin, }{\sl Embeddings of M\"untz spaces : the Hilbertian Case, }{Proc. Amer. Math. Soc., no.6, 141, (2013) MR3034427, 2009-2023.}

\bibitem[We]{We}{D. Werner, }{\sl A remark about M\"untz spaces, }{http://page.mi.fu-berlin.de/werner/preprints/muentz.pdf.}

\bibitem[Wo]{Wo}{P. Wojtaszczyk, }{\sl Banach spaces for analysts, }{Cambridge studies in advanced mathematics, Cambridge, 1991.}






\end{thebibliography}

\end{document}